\documentclass[11pt,oneside,english]{amsart}
\usepackage[latin9]{inputenc}
\usepackage{babel}
\usepackage{enumitem}
\usepackage{amstext}
\usepackage{amsthm}
\usepackage{amssymb}
\usepackage[all]{xy}
\usepackage[unicode=true,pdfusetitle, bookmarks=true,bookmarksnumbered=false,bookmarksopen=false,
 breaklinks=false,pdfborder={0 0 1},backref=false,colorlinks=false]
 {hyperref}
 \usepackage{cleveref}
\usepackage{chngcntr}
\usepackage[title]{appendix}
\makeatletter
\numberwithin{equation}{section}
\numberwithin{figure}{section}
\theoremstyle{plain}
\newtheorem{thm}{\protect\theoremname}[section]
  \theoremstyle{remark}
  \newtheorem{rem}[thm]{\protect\remarkname}
  \theoremstyle{plain}
  \newtheorem{lem}[thm]{\protect\lemmaname}
  \theoremstyle{plain}
  \newtheorem{prop}[thm]{\protect\propositionname}
    \theoremstyle{plain}
  \newtheorem{cor}[thm]{\protect\corollaryname}
    \theoremstyle{definition}
  \newtheorem{defn}[thm]{\protect\definitionname}
 \newlist{casenv}{enumerate}{4}
 \setlist[casenv]{leftmargin=*,align=left,widest={iiii}}
 \setlist[casenv,1]{label={{\itshape\ \casename} \arabic*.},ref=\arabic*}
 \setlist[casenv,2]{label={{\itshape\ \casename} \roman*.},ref=\roman*}
 \setlist[casenv,3]{label={{\itshape\ \casename\ \alph*.}},ref=\alph*}
 \setlist[casenv,4]{label={{\itshape\ \casename} \arabic*.},ref=\arabic*}

\usepackage{multicol}

\makeatother

  \providecommand{\lemmaname}{Lemma}
  \providecommand{\propositionname}{Proposition}
  \providecommand{\remarkname}{Remark}
 \providecommand{\casename}{Case}
\providecommand{\theoremname}{Theorem}
\providecommand{\corollaryname}{Corollary}
  \providecommand{\definitionname}{Definition}
\begin{document}
\global\long\def\Z{\mathbb{\mathbf{Z}}}
\global\long\def\Spec{\mathrm{Spec}}
\global\long\def\G{\mathbb{G}}
\global\long\def\GL{\mathrm{GL}}
\global\long\def\O{\mathcal{O}}
\global\long\def\SL{\mathrm{SL}}
\global\long\def\C{\mathbf{C}}
\global\long\def\tr{\mathrm{tr}}
\global\long\def\Aut{\mathrm{Aut}}
\global\long\def\Out{\mathrm{\mathrm{Out}}}
\global\long\def\Q{\mathbf{Q}}
\global\long\def\R{\mathbf{R}}
\global\long\def\Perms{\mathrm{Perms}}
\global\long\def\Fib{\mathrm{Fib}}
\global\long\def\F{\mathbf{F}}
\global\long\def\Comp{\mathrm{Comp}}

\global\long\def\PGL{\mathrm{PGL}}
\global\long\def\Y{\mathbb{Y}}
\global\long\def\Aut{\mathrm{Aut}}
\global\long\def\sign{\mathrm{sign}}
\global\long\def\PSL{\mathrm{PSL}}
\global\long\def\trace{\mathrm{trace}}

\global\long\def\M{\mathcal{M}}
\global\long\def\rel{\mathrm{rel}}
\global\long\def\F{\mathbb{F}}
\global\long\def\MCG{\mathrm{MCG}}
\global\long\def\Sp{\mathrm{Sp}}
\global\long\def\Hom{\mathrm{Hom}}
\global\long\def\Id{\mathrm{Id}}
\global\long\def\D{\mathcal{D}}
\global\long\def\Mod{\mathrm{Mod}}
\global\long\def\pure{\mathrm{p}}

\global\long\def\RR{\mathfrak{R}}
\global\long\def\PMod{\mathrm{PMod}}
\global\long\def\X{\mathbb{\mathcal{X}}}
\global\long\def\cover{\mathrm{cover}}
\global\long\def\mar{\mathsf{mark}}
\global\long\def\CC{\mathcal{C}}
\global\long\def\newmacroname{}
\global\long\def\Cayley{\mathrm{Cayley}}
\global\long\def\QQ{\mathcal{Q}}
\global\long\def\N{\mathbf{N}}
\global\long\def\garr{\overrightarrow{\gamma}}
\global\long\def\forget{\mathsf{forget}}
\global\long\def\GG{\mathcal{G}}
\global\long\def\A{\mathcal{A}}
\global\long\def\H{\mathcal{H}}
\global\long\def\sc{\mathrm{sc}}
\global\long\def\hol{\mathrm{hol}}
\global\long\def\Stab{\mathrm{Stab}}
\global\long\def\supp{\mathrm{supp}}
\global\long\def\Ave{\mathrm{Ave}}
\global\long\def\Leb{\mathrm{Leb}}

\title{Counting saddle connections in a homology class modulo $q$}

\thanks{*R. R\"{u}hr was supported in part by the S.N.F., project number 168823}

\author{Michael Magee and Rene R\"{u}hr*\\
 with an Appendix by Rodolfo Guti\'{e}rrez-Romo}
\begin{abstract}
We give effective estimates for the number of saddle connections on a translation surface that have length $\leq L$ and  are in a prescribed homology class modulo $q$. Our estimates apply to almost all translation surfaces in a stratum of the moduli space of translation surfaces, with respect to the Masur-Veech measure on the stratum.
\end{abstract}

\maketitle

\section{Introduction}

Let $S$ be a closed topological surface of genus $g\geq2$, and $\Sigma\subset S$
a finite collection of marked points. A\emph{ translation surface
structure} on $(S,\Sigma)$ is a set of complex charts on $S-\Sigma$
such that all transition functions are translations, and such that
the charts extend to conical singularities at $\Sigma$ with cone
angles given by positive integer multiples of $2\pi$. We write $\tilde{S}$
for a translation surface with underlying topological surface $(S,\Sigma)$.
As such, the complex charts give a Euclidean (flat) metric on $S-\Sigma$.
A \emph{saddle connection }on $\tilde{S}$ is a geodesic in $S-\Sigma$
with endpoints in $\Sigma$. 
 For $q\in\Z_{+}$ we write $\Z_{q}\stackrel{\mathrm{def}}{=}\Z/q\Z$.
For any $q\in\Z_{+}$, each saddle connection gives a homology class
in $H_{1}(S,\Sigma,\Z_{q})$. Let $\xi\in H_{1}(S,\Sigma,\Z_{q})$.
In this paper, we give precise estimates for 
\begin{quote}
$N(\tilde{S};L,\xi)=$ The number of saddle connections on $\tilde{S}$
of length $\leq L$ and in the homology class $\xi$.
\end{quote}
This problem is a direct generalization to genus $g\geq2$ of the
following question:
\begin{quote}
\emph{How many primitive lattice points in $\Z^{2}$ are congruent
to $(a,b)$ modulo $q$ and contained in a ball of radius $L$?}
\end{quote}
Our estimates for $N(\tilde{S};L,\xi)$ will be uniform over parameters
$L$ and $q$ and, owing to our use of ergodic methods, apply to almost
all translation surfaces in a sense that will be made precise below.

The issue of estimating $N(\tilde{S};L)$, the number of saddle connections
on $\tilde{S}$ of length $\leq L$, i.e., without any congruence
aspect, is arguably one of the central problems of Teichm\"{u}ller dynamics.
It was proved by Masur \cite{MASURGROWTH,MASURLOWER} that
\begin{thm}[Masur]
Given $\tilde{S}$, there exist $0<c_{1}(\tilde{S})\leq c_{2}(\tilde{S})<\infty$
such that for sufficiently large $L$,
\[
c_{1}\leq\frac{N(\tilde{S};L)}{L^{2}}\leq c_{2}.
\]
\end{thm}
We now fix a function $\kappa:\Sigma\to\Z_{+}$ with the property
that $\sum_{z\in\Sigma}(\kappa(z)-1)=2g-2$. The area of a translation surface is the area inherited from the complex charts.
 The \emph{Teichm\"{u}ller
space }$\X_{\kappa}$ is the collection of unit area translation surface structures
on $S$ with conical singularities at $\Sigma$, with cone angle at
$z\in\Sigma$ given by $2\pi\kappa(z)$, modulo isotopies of $S$
relative to $\Sigma$. The \emph{pure mapping class group} of $(S,\Sigma)$
is the collection of homeomorphisms of $S$ fixing $\Sigma$, up to
isotopy relative to $\Sigma$. It is denoted by $\PMod(S ,\Sigma)$. The group $\PMod(S,\Sigma)$ acts on
$\X_{\kappa}$, and we define
\[
\H(\kappa)\stackrel{\mathrm{def}}{=}\PMod(S,\Sigma)\backslash\X_{\kappa}.
\]
The space $\H(\kappa)$ is called a \emph{stratum} of the moduli space
of unit area translation surfaces; it need not be connected, but the connected
components have been classified by Kontsevich and Zorich \cite{KZ}.
Throughout the rest of the paper, we will view $S$, $\Sigma$, $\kappa$
as fixed, and $\M$ will denote a connected component of $\H(\kappa)$.
We will write $[\tilde{S}]\in\H(\kappa)$ for the modulus of a translation
surface $\tilde{S}$.

There is an $\SL_{2}(\R)$ action on $\X_{\kappa}$ by postcomposition
on the charts of translation structures. This action descends to an
action on $\H(\kappa)$ and $\M$. There is an $\SL_{2}(\R)$-ergodic
and invariant probability measure $\mu$ on $\M$ called the Masur-Veech
measure \cite{MASUR,VEECHGAUSS}. The Masur-Veech measure is in the
Lebesgue class with respect to a natural affine orbifold structure
on $\M$.

Eskin and Masur proved the following theorem in \cite{ESKINMASUR}:
\begin{thm}[Eskin-Masur]
\label{thm:Eskin-Masur}There exists a constant $c=c(\M)$ such that
for $\mu$-almost all $[\tilde{S}]\in\M$,
\[
N(\tilde{S};L)=\pi cL^{2}+o(L^{2})
\]
as $L\to\infty$.
\end{thm}
An $L^{1}$-averaged result of Theorem \ref{thm:Eskin-Masur} had
previously been obtained by Veech \cite{VEECHSIEGEL}. The broad philosophy
behind the works of Veech \cite{VEECHSIEGEL} and Eskin and Masur
\cite{ESKINMASUR} is that one may use dynamics on moduli spaces to
obtain estimates for quantities defined on points in the moduli space.
This philosophy stems from Margulis's work on the Oppenheim conjecture
\cite{MARGULIS}. Indeed, Veech's paper \cite{VEECHSIEGEL} was partly
inspired by a lecture of Margulis, and Theorem \ref{thm:Eskin-Masur}
uses ideas from the paper of Eskin, Margulis and Mozes \cite{EMM}
on the quantitative form of the Oppenheim conjecture.

Recently, Nevo, R\"{u}hr and Weiss \cite{NRW} obtained an improvement
of Theorem \ref{thm:Eskin-Masur} with an effective error term that
saves a power of $L$:
\begin{thm}[Nevo-R\"{u}hr-Weiss]
\label{thm:NRW}There exist constants $c=c(\M)$ and $\kappa=\kappa(\M)$
such that for $\mu$-almost all $[\tilde{S}]\in\M$,
\[
N(\tilde{S};L)=\pi cL^{2}+O(L^{2-\kappa})
\]
as $L\to\infty$. The implied constant depends on $[\tilde{S}]$.
\end{thm}
The new input to Theorem \ref{thm:NRW} that allows Nevo, R\"{u}hr and
Weiss to obtain a power savings error term is an effective ergodic
theorem in the form of a spectral gap for the $\SL_{2}(\R)$ action
on $\M$. The spectral gap for this action was proved by Avila, Gou\"{e}zel
and Yoccoz in \cite{AGY}.

A simplified form of the main theorem of the current paper is
\begin{thm}
\label{thm:main-theorem}There are constants $Q_{0}\in\Z_+$, $c,\eta,\alpha>0$,
 depending on $\M$ such that for $\mu$-almost
all $x\in\M$, for all $\tilde{S}$ such that $[\tilde{S}]=x$, for
all $q$ coprime to $Q_0$, for all $\xi\in H_{1}(S,\Sigma,\Z_{q})$ that
can arise from a saddle connection on $\tilde{S}$, 
\[
N(\tilde{S};L,\xi)=\frac{\pi cL^{2}}{|\PMod(S,\Sigma).\xi|}+O(q^{\alpha}L^{2-\eta}).
\]
The implied constant depends on $[\tilde{S}]$.
\end{thm}
\begin{rem}
The homology class $\xi\in H_{1}(S,\Sigma,\Z_{q})$ either only arises
from saddle connections with distinct endpoints, in which case
\[
|\PMod(S,\Sigma).\xi|=q^{2g}
\]
(cf. Lemmas \ref{lem:transitive-relative} and \ref{lem:num_hom_classes})
or else $\xi$ only arises from saddle connections whose endpoints
coincide, and in this case 
\[
|\PMod(S,\Sigma).\xi|=q^{2g}\prod_{\begin{subarray}{c}
p\text{ prime}\\
p|q
\end{subarray}}\left(1-\frac{1}{p^{2g}}\right)
\]
(cf. Theorem \ref{thm:transitive-on-unimodular} and Lemma \ref{lem;num-unimod}).
\end{rem}
\begin{rem}
The main term of Theorem \ref{thm:main-theorem} is larger than the
error term whenever 
\[
L\gg_{[\tilde{S}]}q^{\frac{2g+\alpha}{\eta}}.
\]
The exponent of $q$ here hence depends on $g$, $\alpha$,
and $\eta$. Our proof shows that $\alpha$ can be taken to be $D+2$
where $D$ is the dimension of the group $\Z_{q}^{2g}\rtimes\Sp_{2g}(\Z_{q})$
in the case that $\xi$ comes from saddle connections with distinct
endpoints, or $\Sp_{2g}(\Z_{q})$ otherwise. So the exponent $\alpha$
can be made explicit, but we expect this not to be optimal. The constant
$\eta$ depends on the size of the uniform spectral gap in Theorem
\ref{thm:Uniform-spectral-gap} below. We have not tried to optimize
this dependency, since the size of the uniform spectral gap, although
theoretically could be made explicit, would be very far from optimal
if one tried to do so. 
\end{rem}

\subsection{An overview of the proof of Theorem \ref{thm:main-theorem}}
 
 We state a more refined version of Theorem \ref{thm:main-theorem} in Theorem \ref{thm:tech} below. This involves counting tuples of saddle connections, with an extra restraint
 that the saddle connections form a given \emph{configuration} $\CC$. Configurations are defined in Section  \ref{holo}.
 This leads to a more refined counting function $N^{\CC}(x;L,\xi)$ where $x\in \M$, $L\geq0$ and $\xi\in H_1(S,\Sigma,\Z_q)$.
 
 The starting observation is that $N^{\CC}(x;L,\xi)$, for fixed $L$, lifts  naturally to a covering space $\M(\Theta_q^\sigma)$ of $\M$. Here $\sigma$ is a parameter 
 encoding the location of the endpoints of the saddle connections of configuration $\CC$.
 These covering spaces are defined in Section  \ref{subsec:Connected-components-of}.
 
 Now working on the covering space $\M(\Theta_q^\sigma)$, we wish to apply the arguments of Nevo, R\"{u}hr, and Weiss from \cite{NRW}. 
 This argument requires in particular two key estimates. The first is a pointwise ergodic theorem for the action of $\SL_2(\R)$ on a Sobolev space of functions on $\M(\Theta_q^\sigma)$. The second is the Siegel-Veech formula. Given these two ingredients, our arguments are the same is in \cite{NRW}, with a minor streamlining due to recent work of Athreya, Cheung, and Masur \cite{ACM}. These arguments are given in Section \ref{proof}.
 
 However, the estimates we obtain for $N^{\CC}(x;L,\xi)$ using this argument will not have an explicit dependence on $q$ unless we can both
 \begin{itemize}
 \item Make the pointwise ergodic theorem uniform in $q$, and
 \item Make the dependence of the Siegel-Veech formula on $q$ explicit.
 \end{itemize}
 
 The description of the Siegel-Veech formula at level $q$ goes along the same lines as for the usual Siegel-Veech formula. The issue that arises is the description of the \emph{Siegel-Veech constants} at level $q$. The key fact that allows us to do this, is that for any configuration, there is only one orbit of the mapping class group on $H_1(S,\Sigma,\Z_q)$ that can arise from that configuration. For closed saddle connections, and when $q$ is prime, this fact is a result of Witt's Theorem on symplectic transformations over finite fields. For general $q$ and $\CC$, we state the results we need on orbits in $H_1(S,\Sigma,\Z_q)$ in Section \ref{subsec:properties-of-the-groups}. Given these ingredients, the treatment of the Siegel-Veech formula at level $q$ is given in Section \ref{sec-SV}. 
 
 We now turn to the pointwise ergodic theorem, Theorem \ref{thm:pointwise-ergodic}. The proof of the pointwise ergodic theorem for $\M$ given by Nevo, R\"{u}hr, and Weiss in \cite{NRW} uses the spectral gap for the $\SL_2(\R)$ action on $\M$ proved by Avila, Gou\"{e}zel, and Yoccoz \cite{AGY}. The same arguments prove the $q$-uniform pointwise ergodic theorem that we need, provided that one can prove the $\SL_2(\R)$ actions on $\M(\Theta_q^\sigma)$ have a \emph{uniform spectral gap}. 
 
 Before discussing the spectral gap, we explain what it means. There are several different ways to state the result (see \cite[Introduction]{MAGEEKZ} for a discussion), but the one we will use here is in terms of representation theory. There is a family of irreducible unitary representations of $\SL_2(\R)$ called \emph{complementary series representations}. They are parametrized by a parameter $u\in(0,1)$ and written $\Comp^u$. The action of $\SL_2(\R)$ on $\M(\Theta_q^\sigma)$ has a spectral gap $\eta>0$ if
 the abstract decomposition of the unitary representation $L^2(\M(\Theta_q^\sigma))$ as a direct integral is supported away from $\Comp^u$ with $u\in (1-\eta,1)$.
 For a family of $q$, there is a uniform spectral gap if $\eta$ can be taken to be the same for all $q$.

 A uniform spectral gap result for the $\M(\Theta_q^\sigma)$ that arise, in the current context, from configurations of closed saddle connections, was obtained by Magee in \cite{MAGEEKZ}. Establishing this uniform spectral gap for the covering spaces arising from non-closed saddle connections is a main point of this paper. We extend the results of \cite{MAGEEKZ} to congruence covers of $\M$ that arise from relative homology. This extended result is given in Theorem \ref{thm:Uniform-spectral-gap}.
 
 We do this by establishing two facts that can be inserted into the framework of \cite{MAGEEKZ} as `black boxes' to obtain the extension. The framework of \cite{MAGEEKZ} uses Veech's zippered rectangles construction. This associates to each $\M$ a collection of Rauzy-Veech monoids. These are defined in Section \ref{subsec:Rauzy-Veech-monoids-and}. They can be identified with submonoids of the pure mapping class group of $(S,\Sigma)$. Each of these act on a piece of $H_1(S,\Sigma,\Z_q)$ and generate a finite group $J_{\sigma}(q)$. There is also a group $G_\sigma(q)$ that is the whole automorphism group of the relevant piece of  $H_1(S,\Sigma,\Z_q)$. In other words, $G_\sigma(q)$ is the largest group that $J_{\sigma}(q)$ could be. 
 
 The first black box we need is that there are fixed generators of the Rauzy-Veech monoid, such that the associated Cayley graphs of $G_\sigma(q)$ are uniform expander graphs. This is stated as {\bf Property I: Uniform Expansion} below. 
 It is not possible for this to hold unless $J_\sigma(q) = G_\sigma(q)$ for each $q$ in the family. This `Strong Approximation' statement is absolutely essential to the current paper. It is given precisely in Proposition \ref{prop:Strong-Approximation} and proved by Guti\'{e}rrez-Romo in the Appendix. Once it is established, the uniform expansion of the associated Cayley graphs relies ultimately on Kazhdan's property (T) for $\Z^{2g} \rtimes \Sp_{2g}(\Z)$ when $g\geq 2$, which is a result of Burger \cite{BURGER}. The necessary arguments are given in Section \ref{sec:bb1}.
 
 The remaining black box we need is a lower bound for the dimension of `new' representations of $G_{\sigma}(q)$ that is polynomial in $q$. This is stated in {\bf Property II: Quasirandomness} below. The proof involves exploiting the semidirect product structure of $G_\sigma(q)$ together with known results about the corresponding groups in absolute homology. These arguments are given in Section \ref{sec:bb2}.

\subsection{Acknowledgments}
We would like to thank Barak Weiss for the suggestion that the spectral gap result of Magee might be relevant to counting problems on translation surfaces. This led to a meeting in Tel Aviv University that instigated the work on this paper.

\subsection{Notation}
We collect all the special notation that we use here.
We write $\Z_{q}=\Z/q\Z$. We use the notation $\tilde{H}$ for reduced homology. If $R$ is an
abelian group then $H_{1}(R)=H_{1}(S,R)$ and $H_{1}^{\rel}(R)=H_{1}(S,\Sigma,R)$.

\section{Covering spaces and uniform spectral gap\label{sec:Covering-spaces-and}}

\subsection{Connected components of strata of translation surfaces and covering
spaces \label{subsec:Connected-components-of}}

Let $G$ be a finite group and let 
\[
\Theta:\PMod(S,\Sigma)\to G
\]
be a surjective group homomorphism. Then we can consider $Y(\Theta)\stackrel{\mathrm{def}}{=}\ker(\Theta)\backslash\X_{\kappa}$.
There is a natural map 
\begin{equation}
Y(\Theta)\xrightarrow{\cover_{\Theta}}\H(\kappa)\label{eq:covering map of strata}
\end{equation}
with deck transformation group $G$. The map $\cover_{\Theta}$ is
a covering map of affine orbifolds. We define
\[
\M(\Theta)\stackrel{\mathrm{def}}{=}\cover_{\Theta}^{-1}(\M).
\]

The $\SL_{2}(\R)$ action on $\X_{\kappa}$ also gives an action on
$\M(\Theta)$ that preserves $\mu_{\M(\Theta)}$, the pull-back of
$\mu_{\M}$ with respect to the map $\cover_{\Theta}$ and the counting
measure on the fibres of $\cover_{\Theta}$.

The goal of this section is to prove that for certain specific $\Theta$
related to our counting problem, the $\SL_{2}(\R)$ actions on $\M(\Theta)$
have a uniform spectral gap. We now discuss the maps $\Theta$ that
we are interested in. Let $\Gamma=\PMod(S,\Sigma)$. We are interested
in the rings $\Z_{q}=\Z/q\Z$ and $\Z$. So that we can discuss these
at the same time, we use the notation $\Z_{\infty}=\Z$ and convention
that $q|\infty$ for all finite $q$. In the following, $q$ is either
a finite natural number $\geq2$ or $\infty$.

For each such $q$ we have a short exact sequence
\begin{equation}
0\to H_{1}(\Z_{q})\to H_{1}^{\rel}(\Z_{q})\xrightarrow{\delta}\tilde{H}_{0}(\Sigma,\Z_{q})\to0\label{eq:homology-SES}
\end{equation}
where $\delta$ is the connecting map from the long exact sequence
of homology with $\Z_{q}$ coefficients for the pair $(S,\Sigma)$.
We use the notation $\tilde{H}_{*}$ for reduced homology. Henceforth
we will think of $H_{1}(\Z_{q})$ as a subgroup of $H_{1}^{\rel}(\Z_{q})$
using the second arrow of (\ref{eq:homology-SES}). The short exact
sequence (\ref{eq:homology-SES}) is composed of $\Z[\Gamma]$-modules,
where the module structure on $H_{1}(\Z_{q})$ and $H_{1}^{\rel}(\Z_{q})$
comes from the action of $\Gamma$ on first homology, and the module
structure on $\tilde{H}_{0}(\Sigma,\Z_{q})$ comes from the trivial
action of $\Gamma$. Then each map in (\ref{eq:homology-SES}) is
a $\Z[\Gamma]$-module homomorphism. 

If we fix an element $\sigma\in\tilde{H}_{0}(\Sigma,\Z)$, then for
each $q$ we obtain by applying the map $\tilde{H}_{0}(\Sigma,\Z)\to\tilde{H}_{0}(\Sigma,\Z_{q})$
an element $\sigma_{q}\in\tilde{H}_{0}(\Sigma,\Z_{q})$. We then consider
the submodule $\Z_{q}\sigma_{q}\subset\tilde{H}_{0}(\Sigma,\Z_{q})$.
The preimage of this submodule under $\delta$ will play an important
role in this paper so we give it a name:
\[
H_{1}^{\sigma}(\Z_{q})\stackrel{\mathrm{def}}{=}\delta^{-1}\left(\Z_{q}\sigma_{q}\right)\subset H_{1}^{\rel}(\Z_{q}).
\]
Notice one special case is if $\sigma=0$ then $H_{1}^{\sigma}(\Z_{q})=\ker(\delta)=H_{1}(\Z_{q})$.
We obtain a map
\[
\Theta_{q}^{\sigma}:\PMod(S,\Sigma)\to\Aut(H_{1}^{\sigma}(\Z_{q}))
\]
and define $G_{\sigma}(q)$ to be the image of $\Theta_{q}^{\sigma}$.
Hence we obtain covering spaces $\M(\Theta_{q}^{\sigma})$ of $\M$. 

The main theorem of this section is the following.
\begin{thm}
\label{thm:Uniform-spectral-gap}Let $\sigma\in\tilde{H}_{0}(\Sigma,\Z)$
be given by $\sigma=z_{1}-z_{2}$ where $z_{1},z_{2}\in\Sigma$. Then there is $Q_0 \in \Z_+$ such that 
for $q$ coprime to $Q_0$, the family of covering spaces $\M(\Theta_{q}^{\sigma})$
have a uniform spectral gap.
\end{thm}

Theorem \ref{thm:Uniform-spectral-gap} was proved when $\sigma=0$
(i.e. $z_{1}=z_{2}$) by Magee in \cite{MAGEEKZ}, using the work of Guti\'{e}rrez-Romo
\cite[Theorem 1.1]{GR} as a crucial input. \emph{Throughout the rest of this Section
\ref{sec:Covering-spaces-and} we therefore assume $\sigma=z_{1}-z_{2}\neq0$.}

\subsection{Rauzy-Veech monoids and groups\label{subsec:Rauzy-Veech-monoids-and}}

Let $\M$ denote a connected component of a stratum of unit area abelian
differentials for $\left(S,\Sigma\right)$. Following Veech \cite{VEECHGAUSS}
we can associate to $\M$ a finite labeled directed graph called a
\emph{Rauzy diagram.}

Here we give some background on Rauzy diagrams. Let $\A$ be a finite
alphabet with $|\A|=d\geq3$. Our fundamental combinatorial objects
are pairs
\[
\pi=\left(\begin{array}{c}
\pi_{t}\\
\pi_{b}
\end{array}\right)=\left(\begin{array}{cccc}
\alpha_{t,1} & \alpha_{t,2} & \ldots & \alpha_{t,d}\\
\alpha_{b,1} & \alpha_{b,2} & \ldots & \alpha_{b,d}
\end{array}\right)
\]
where each of $\pi_{t}$ and $\pi_{b}$ are orderings of the elements
of $\A$. The letters `t' and `b' stand for `top' and `bottom'. We
think of both $\pi_{t}$ and $\pi_{b}$ as rows of elements of $\A$,
with each element appearing in one place of the row. We say that $\pi$
is \emph{irreducible }if there is no $1\leq j<d$ such that $\{\alpha_{t,1},\ldots,\alpha_{t,j}\}=\{\alpha_{b,1},\ldots,\alpha_{b,j}\}$.
We say that $\pi$ is \emph{degenerate }if any of the following hold
for $1\leq j<d$:
\begin{itemize}
\item $\alpha_{t,j}$ is the last entry of the bottom row, $\alpha_{t,j+1}$
is the first entry of the bottom row, and $\alpha_{t,1}$ directly
follows $\alpha_{t,d}$ in the bottom row. 
\item $\alpha_{t,j+1}$ is the first entry of the bottom row and $\alpha_{t,1}$
directly follows $\alpha_{t,j}$ in the bottom row.
\item $\alpha_{t,j}$ is the last entry of the bottom row and $\alpha_{t,j+1}$
directly follows $\alpha_{t,d}$ in the bottom row.
\end{itemize}
Otherwise we say $\pi$ is \emph{nondegenerate. }For example, if $\A=\{A,B,C,D\}$
then 
\[
\left(\begin{array}{cccc}
A & B & C & D\\
B & D & A & C
\end{array}\right)
\]
is an irreducible nondegenerate pair.

There are two natural operations on pairs $\pi$ of permutations that
arise naturally from the Rauzy induction algorithm. The operations
are called `top' and 'bottom' operations. The top operation on $\pi$
modifies the bottom row by moving the occurrence of $\alpha_{b,d}$
to the immediate right of the occurrence of $\alpha_{t,b}$. The bottom
operation modifies the top row by moving $\alpha_{t,d}$ to the right
of $\alpha_{b,d}$. As in \cite{AGY} we say that the last element
of the unchanged row is the \emph{winner }and the last element of
the row of $\pi$ that is to be changed the \emph{loser. }One can
check that if $\pi$ is irreducible and nondegenerate then so are
any pairs obtained from applying top and bottom operations.

For fixed $\A$, the connected components of the directed graph whose
vertices are pairs of permutations and edges given by top and bottom
moves are called \emph{Rauzy diagrams. }The vertices of a Rauzy diagram
are called a \emph{Rauzy class.} Veech's theory of zippered rectangles
\cite{VEECHGAUSS} associates to each Rauzy class a connected component
$\M$ of a stratum of the moduli space of translation surfaces. The
relationship between $\M$ and $\RR(\M)$ is that for each $\pi\in\RR(\M)$,
there is a multitude (depending on a parameter space) of ways to build
a translation surface with modulus in $\M$ according to Veech's \emph{zippered
rectangles construction}.\emph{ }We will not need the full details
of the zippered rectangles construction in this paper, so omit it.
Veech proved that every such connected component $\M$ arises in this
way from a Rauzy diagram $\RR$, and moreover we can and will assume
that the elements of $\RR$ are irreducible and nondegenerate. We
fix for each $\M$ a choice of such $\RR=\RR(\M)$ throughout the
rest of the paper.

Using the zippered rectangles construction, Yoccoz \cite[\S 9.2]{YOCCOZ}
associates to each $\pi\in\RR$ a canonical translation surface $\tilde{S}_{\pi}$.
In what follows it is the underlying topological surface $S_{\pi}$
of this canonical translation surface that will be relevant. The surface
$S_{\pi}$ comes with a set of marked points $\Sigma_{\pi}$ that
coincides with the conical singularities of $\tilde{S}_{\pi}$. 

We now define \emph{Rauzy-Veech monoids}. Rauzy-Veech monoids are
usually considered at the level of (absolute) homology, but can be
defined at the level of (relative) homotopy, following Avila-Matheus-Yoccoz
\cite[\S 4.1]{AMY}. To each arrow of $\gamma:\pi\to\pi'$ of $\D$
(\emph{loc. cit.)} build an orientation preserving homeomorphism $H_{\gamma}:(S_{\pi},\Sigma_{\pi})\to(S_{\pi'},\Sigma_{\pi'})$.
Moreover it is possible to name all the elements of all $\Sigma_{\pi}$
for $\pi\in\RR$ in such a way that each $H_{\gamma}$ respects the
namings. We write $[H_{\gamma}]$ for the isotopy class of this homeomorphism
relative to $\Sigma_{\pi}$. If $L$ is any directed loop in $\RR$
beginning and ending at $\pi$ then we get by composing the $[H_{\gamma}]$
along the loop an element $[H_{L}]$ in the pure mapping class group
$\PMod(S_{\pi},\Sigma_{\pi}).$ Composing directed loops yields a
monoid and the image of this monoid under $L\mapsto[H_{L}]$ is a
submonoid $\Lambda_{\pi}\leq\PMod(S_{\pi},\Sigma_{\pi})$ that we
call the \emph{(pure) modular Rauzy-Veech monoid}.

\subsection{Input for the spectral gap}

The methods of \cite{MAGEEKZ} for proving Theorem \ref{thm:Uniform-spectral-gap}
for $\sigma=0$ relies on `black-box' results about the covering spaces
$\M(\Theta_{q}^{\sigma})$. As long as these can be checked when $\sigma\neq0$,
the methods of \cite{MAGEEKZ} will also apply to prove Theorem \ref{thm:Uniform-spectral-gap}
for $\sigma\neq0$. We explain the necessary inputs in this section. 

Again, let $\Gamma=\PMod(S,\Sigma)$. Each $H_{1}^{\sigma}(\Z_{q})$
is a $\Z[\Gamma]$-submodule of $H_{1}^{\rel}(\Z_{q})$ that fits
into the short exact sequence of $\Z[\Gamma]$-modules
\begin{equation}
0\to H_{1}(\Z_{q})\to H_{1}^{\sigma}(\Z_{q})\xrightarrow{\delta}\Z_{q}\sigma_{q}\to0.\label{eq:SES-of-modules}
\end{equation}
Moreover it is important to note that if $q_{1}|q_{2}$ with $q_{1},q_{2}\in\{2,\ldots,\infty\}$
then there is a commutative diagram of $\Z[\Gamma]$-modules
\begin{equation}
\xymatrix{0\ar[r]\ar[d] & H_{1}(\Z_{q_{2}})\ar[r]\ar[d] & H_{1}^{\sigma}(\Z_{q_{2}})\ar[r]\ar[d] & \tilde{H}_{0}(\Sigma,\Z_{q_{2}})\ar[r]\ar[d] & 0\ar[d]\\
0\ar[r] & H_{1}(\Z_{q_{1}})\ar[r] & H_{1}^{\sigma}(\Z_{q_{1}})\ar[r] & \tilde{H}_{0}(\Sigma,\Z_{q_{1}})\ar[r] & 0
}
\label{eq:commutative}
\end{equation}
where the vertical maps are surjections coming from tensor product
with $\Z_{q_{1}}$ over $\Z$. Recall we previously defined $G_{\sigma}(q)$
to be the image of the action of $\Gamma$ on $H_{1}^{\sigma}(\Z_{q})$.

Let $\pi$ be a vertex of a irreducible, non-degenerate Rauzy diagram
corresponding to a connected component $\M$ of a stratum of unit
area abelian differentials for $(S,\Sigma)$. Fixing a homeomorphism
$(S,\Sigma)\to(S_{\pi},\Sigma_{\pi})$ we can view the $\Theta_{q}^{\sigma}$
as homomorphisms $\Theta_{q}^{\sigma}:\PMod(S_{\pi},\Sigma_{\pi})\to G_{\sigma}(q)$.

The necessary inputs to the methods of \cite{MAGEEKZ} are the following
two properties.
\begin{description}
\item [{Property~I.~Uniform~Expansion}] Recall $\Lambda_{\pi}$ is a
Rauzy-Veech monoid as defined in $\S$\ref{subsec:Rauzy-Veech-monoids-and}.
We require that there is some element $\pi\in\RR$, and a fixed finite
set of elements $T\subset\Lambda_{\pi}$ such that the family of Cayley
graphs $\Cayley(G_{\sigma}(q),\Theta_{q}^{\sigma}(T\cup T^{-1}))$
have a uniform spectral gap; that is, the second largest eigenvalue
of the adjacency operator is $\leq2|T|-\epsilon$ for some $\epsilon>0$
independent of $q.$
\item [{Property~II.~Quasirandomness}] We say a representation of $G_{\sigma}(q)$
is \emph{new }if it doesn't factor through any representation of $G_{\sigma}(q_{1})$
for some $q_{1}|q$. We then require that there are some constants
$C,\eta>0$ such that if $(\rho,V)$ is a new nontrivial irreducible
representation of $G_{q}$ that $\dim(V)\geq C|G(q)|^{\eta}$.
\end{description}
Once these two properties have been established, Theorem \ref{thm:Uniform-spectral-gap}
is proved, mutatis mutandis, exactly as in the case of $\sigma=0$
that was treated in \cite{MAGEEKZ}.

\subsection{Properties of $G_{\sigma}(q)$\label{subsec:properties-of-the-groups}}

Recall we assume $\sigma=z_{1}-z_{2}\neq0$ and write $\Gamma=\PMod(S,\Sigma)$.
Let $G(q)\leq\Aut(H_{1}(\Z_{q}))$ be the image of the action of $\Gamma$
on $H_{1}(\Z_{q})$. We obtain a short exact sequence of groups 
\begin{equation}
1\to U_{\sigma}(q)\xrightarrow{\iota}G_{\sigma}(q)\xrightarrow{\beta}G(q)\to1\label{eq:group-SES}
\end{equation}
where $U_{\sigma}(q)\leq G_{\sigma}(q)$ is the subgroup of $G_{\sigma}(q)$
that acts trivially on $H_{1}(\Z_{q})$. 

The first issue to address is to describe $G(q)$ and $G_{\sigma}(q)$,
since at this point, they have just been defined as the image of certain
homomorphisms. First we will note in Lemmas \ref{lem:G(q)} and \ref{lem:unipotent-identification}
that for $q$ odd, $U_{\sigma}(q)$ and $G(q)$ are as large as they
can be, and then in Lemma \ref{lem:G_sigma is a semidirect product}
we will describe the group structure of $G_{\sigma}(q)$.

Note that each group $H_{1}(\Z_{q})$ comes with a symplectic bilinear
form $\cap$ given by cap (intersection) product. 
\begin{lem}
\label{lem:G(q)}For $q$ odd or $\infty$, the action of $G(q)$
on $H_{1}(\Z_{q})$ induces an isomorphism $G(q)\cong\Sp(H_{1}(\Z_{q}),\cap)$.
\end{lem}
\begin{proof}
This follows from the fact that $\Gamma$ maps onto $\Sp( H_1(\Z) ,\cap)$ and that $\Sp( H_1(Z) , \cap)$ maps onto $\Sp( H_1(\Z_q) , \cap)$ for all odd $q$ by strong approximation for $\Sp_{2g}(\Z)$. 
\end{proof}

\begin{lem}
\label{lem:unipotent-identification}Let $q$ be odd or $\infty$.
Let $\omega\in H_{1}^{\sigma}(\Z_{q})$ be such that $\delta(\omega)=\sigma_{q}$.
The map $\Phi_{\omega}:g\mapsto g\omega-\omega$ defines a group isomorphism
from $U_{\sigma}(q)$ to $H_{1}(\Z_{q})$. This isomorphism only depends
on $\sigma$ and not the choice of $\omega$. Therefore we may write
$\Phi_{q}=\Phi_{\omega}$ for the canonically obtained map. Given
an element $\alpha$ of $H_{1}(\Z_{q})$, $\Phi_{q}^{-1}(\alpha)\in U_{\sigma}(q)$
is characterized by
\[
\Phi_{q}^{-1}(\alpha)[\omega]=\omega+\alpha
\]
for all $\omega\in\delta^{-1}(\sigma_{q})$.
\end{lem}
\begin{proof}
The image of $\Phi_{\omega}$ is in $H_{1}(\Z_{q})$ since for any
$g\in U_{\sigma}(q)$, $\delta(g\omega-\omega)=g.\delta(\omega)-\delta(\omega)=0$,
so $g\omega-\omega\in\ker(\delta)=H_{1}(\Z_{q})$.

The map $\Phi_{\omega}$ is a homomorphism because $g_{1}g_{2}\omega-\omega=g_{1}(g_{2}\omega-\omega)+g_{1}\omega-\omega=g_{2}\omega-\omega+g_{1}\omega-\omega$;
here we used that $g_{2}\omega-\omega$ is in $H_{1}(\Z_{q})$ and
that $g_{1}$ acts trivially on $H_{1}(\Z_{q})$ when it is in $U_{\sigma}(q)$.

The map $\Phi_{\omega}$ is injective since if $g\omega-\omega=0$,
i.e. $g$ is in the kernel of $\Phi_{\omega}$, then $g$ fixes $\omega$.
Since from \eqref{eq:SES-of-modules} $H_{1}^{\sigma}(\Z_{q})$ is generated
by $\omega$ and $H_{1}(\Z_{q})$, and $g$ acts trivially on $H_{1}(\Z_{q})$,
$g$ acts trivially on $H_{1}^{\sigma}(q)$, so $g$ is the identity.

If $\omega'$ is another element of $H_{1}^{\rel}(\Z_{q})$ such that
$\delta(\omega')=\sigma_{q}$ then $\omega'=\omega+\beta$ with $\beta\in H_{1}(\Z_{q})$.
Then $\Phi_{\omega'}(g)=g\omega'-\omega'=g\omega-\omega+g\beta-\beta$
but for $g\in U_{\sigma}(q)$, $g\beta-\beta=0$. Therefore $\Phi_{\omega}$
doesn't depend on the choice of $\omega$.

To prove that $\Phi_{q}$ is surjective, first let $A$ be an arc
in $S$ that has endpoints $z_{1}$ and $z_{2}$ and is otherwise
disjoint from $\Sigma$, with no self intersections. Then the class
of $[A]\in H_{1}^{\rel}(\Z_{q})$ is such that $\delta([A])=\sigma_{q}$.
Therefore we may use $\omega=[A]$ in the definition of $\Phi_{q}$.
Let $\gamma$ be a nonseparating simple closed curve in $S$. After
an isotopy, we may assume $\gamma$ crosses $A$ in exactly one point.
Then a Dehn twist $D_{\gamma}\in\Gamma$ maps $D_{\gamma}([A])=[A]+[\gamma]$
where $[\gamma]$ is the class of $\gamma$ in $H_{1}(\Z_{q})$. On
the other hand, we may isotope $\gamma$ to obtain $\gamma'$ that
is disjoint from $A$. Then $D_{\gamma'}([A])=[A]$. Furthermore,
$D_{\gamma}D_{\gamma'}^{-1}$ acts trivially on $H_{1}(\Z_{q})$ and
$D_{\gamma}D_{\gamma'}^{-1}([A])=[A]+[\gamma]$. This proves all $[\gamma]\in H_{1}(\Z_{q})$
arising from nonseparating simple closed curves are in the image of
$\Phi_{q}$. To conclude, there is a standard basis of $H_{1}(\Z_{q})$
given by nonseparating simple closed curves, so $\Phi$ maps onto
$H_{1}(\Z_{q})$.
\end{proof}
\begin{lem}
\label{lem:G_sigma is a semidirect product}For $q$ odd, there is
a non-canonical isomorphism $G_{\sigma}(q)\cong H_{1}(\Z_{q})\rtimes G(q)\cong H_{1}(\Z_{q})\rtimes\Sp(H_{1}(\Z_{q}),\cap)$.
\end{lem}
\begin{proof}
Recall that $\beta$ is the surjection from (\ref{eq:group-SES}).
We define for some fixed $\omega\in\delta^{-1}(\sigma_{q})$
\[
\alpha:G_{\sigma}(q)\to H_{1}(\Z_{q})\rtimes G(q),\quad\alpha(g)\stackrel{\mathrm{def}}{=}(g\omega-\omega,\:\beta(g)).
\]
First we check $\alpha$ is a homomorphism. Indeed, 
\begin{align*}
\alpha(g_{1})\alpha(g_{2}) & =(g_{1}\omega-\omega,\:\beta(g_{1}))(g_{2}\omega-\omega,\:\beta(g_{2}))\\
 & =(g_{1}\omega-\omega+\beta(g_{1}).(g_{2}\omega-\omega),\:\beta(g_{1})\beta(g_{2}))\\
 & =(g_{1}g_{2}\omega-\omega,\:\beta(g_{1}g_{2}))=\alpha(g_{1}g_{2}).
\end{align*}
The map $\alpha$ is injective since if $\alpha(g)=(0,e)$ then $\alpha$
fixes $\omega$ and acts trivially on $H_{1}(\Z_{q})$, so acts trivially
on $H_{1}^{\sigma}(\Z_{q})$. Finally, $\alpha$ is surjective, since
clearly $\alpha(U_{\sigma}(\infty))=\{(\Phi(u),e)\::\:u\in U_{\sigma}(q)\}=H_{1}(\Z_{q})\rtimes\{e\}$
by Lemma \ref{lem:unipotent-identification}, and also clearly the
image of $\alpha$ projects onto $G(q)$, since $\beta$ is onto.
These two facts imply $\alpha$ is surjective.
\end{proof}
Finally in this section we make some observations about the orbits
of $G(q)$ and $G_{\sigma}(q)$ on $H_{1}(\Z_{q})$ and $H_{1}^{\sigma}(\Z_{q})$
that will be important later in our proof of a Siegel-Veech formula
(Theorem \ref{thm:siegelveechformula}). We will say that an element
$v$ of a free $\Z_{q}$-module $A$ is \emph{unimodular} if when
$b_{1},\ldots,b_{r}$ are a basis of $A$, and $v=\sum_{i=1}^{r}v_{i}b_{i}$
with $v_{i}\in\Z_{q}$, we have $(v_{1},\ldots,v_{r})=A$. For example,
any nonseparating simple closed curve on $S$ gives rise to a unimodular
element of $H_{1}(\Z_{q})$. The following incarnation of Witt's Theorem
is a direct application of \cite[Theorem 2.8]{KMcD}.
\begin{thm}[Witt's Theorem]
\label{thm:transitive-on-unimodular}For odd $q$ the group $G(q)\cong\Sp(H_{1}(\Z_{q}),\cap)$
acts transitively on unimodular vectors in $H_{1}(\Z_{q})$.
\end{thm}
\begin{lem}
\label{lem;num-unimod}For odd $q$, the number of unimodular vectors
in $H_{1}(\Z_{q})$ is 
\begin{align*}
q^{2g}\prod_{\begin{subarray}{c}
p\text{ prime}\\
p|q
\end{subarray}}\left(1-\frac{1}{p^{2g}}\right).
\end{align*}
\end{lem}
\begin{proof}
If $q=\prod_{p\text{ prime}}p^{e(p)}$ then $H_{1}(\Z_{q})\cong\prod_{p\text{ prime}}H_{1}(\Z_{p^{e(p)}})$.
Under this identification, $v=(v_{p})_{p\text{ prime}}$ is unimodular
if and only if each $v_{p}$ is not in $pH_{1}(\Z_{p^{e(p)}})$. Therefore
there are $\prod_{p\text{ prime}}p^{2ge(p)}-p^{2g(e(p)-1)}=q^{2g}\prod_{\begin{subarray}{c}
p\text{ prime}\\
p|q
\end{subarray}}\left(1-\frac{1}{p^{2g}}\right)$ unimodular vectors in $H_{1}(\Z_{q})$.
\end{proof}
Now we turn to orbits of $G_{\sigma}(q)$ on $H_{1}^{\sigma}(\Z_{q})$.
Recall $\delta$ is the connecting homomorphism from (\ref{eq:homology-SES}).
\begin{lem}
\label{lem:transitive-relative}Let $q$ be odd. Suppose $\sigma=z_{1}-z_{2}\neq0$.
Then $G_{\sigma}(q)$ acts transitively on $\delta^{-1}(\sigma_{q})\subset H_{1}^{\sigma}(\Z_{q})$.
\end{lem}
\begin{proof}
If $v_{1}$ and $v_{2}$ are in $\delta^{-1}(\sigma)$ then $v_{2}-v_{1}$
is in the kernel of $\delta$, hence can be regarded as an element
$\alpha\in H_{1}(\Z_{q})$. Now from Lemma \ref{lem:unipotent-identification}
there is an element $g_{\alpha}\in U_{\sigma}(q)$ such that $g_{\alpha}(v_{1})=v_{1}+\alpha=v_{2}$.
\end{proof}
\begin{lem}
\label{lem:num_hom_classes}If $\sigma=z_{1}-z_{2}\neq0$ then the
number of elements in $\delta^{-1}(\sigma_{q})$ is $|H_{1}(\Z_{q})|=q^{2g}$.
\end{lem}
\begin{proof}
From (\ref{eq:homology-SES}), $\delta^{-1}(\sigma_{q})$ has a simply
transitive action of $H_{1}(\Z_{q})$.
\end{proof}

\subsection{Establishing the inputs to the uniform spectral gap I: Uniform expansion
of Cayley graphs.}\label{sec:bb1}

In this section we check the inputs to Theorem \ref{thm:Uniform-spectral-gap}
hold for the families $\{G_{\sigma}(q)\}$ where $\sigma=\sigma(\CC)$
as before. Recall $\Lambda_{\pi}$ is a Rauzy-Veech monoid defined
in $\S$\ref{subsec:Rauzy-Veech-monoids-and} where $\pi$ is a vertex
of a irreducible, non-degenerate Rauzy diagram corresponding to $\M$.
Our arguments depend on the following theorem that is the resolution
of the `Zorich conjecture' \cite[Appendix A.3 Conjecture 5]{ZORICHLEAVES}.
Recall $\Theta_{\infty}^{0}:\PMod(S_{\pi},\Sigma_{\pi})\to G(\infty)=\Sp(H_{1}(\Z),\cap)$
is the homomorphism obtained from the action of the mapping class
group on absolute homology.
\begin{thm}[{Avila-Matheus-Yoccoz \cite[Theorem 1.1]{AMY} ($\M$ hyperelliptic),
\\
Guti\'{e}rrez-Romo \cite[Theorem 1.1]{GR}}]
\label{thm:Zorich-conj} The monoid $\Theta_{\infty}^{0}(\Lambda_{\pi})$
generates a finite index subgroup of $\Sp(H_{1}(\Z),\cap)$.
\end{thm}
We will also rely crucially on the following strong approximation
property for the Rauzy-Veech monoid that is proved by Guti\'{e}rrez-Romo
in Appendix \ref{ap}.
\begin{prop}[Strong Approximation]
\label{prop:Strong-Approximation}For each connected component $\M$
of a stratum, there is $\pi\in\RR(\M)$ such that for all positive
odd integers $q$, $\Lambda_{\pi}$ maps onto $G_{\sigma}(q)$ under
the `action on mod q relative homology' map $\Theta_{q}^{\sigma}:\PMod(S_{\pi},\Sigma_{\pi})\to G_{\sigma}(q)$.
\end{prop}
\begin{rem}
Proposition \ref{prop:Strong-Approximation} is known when $\sigma=0$.
In this case $G_{\sigma}(\infty)=G(\infty)=\Sp(H_{1}(\Z),\cap)$ and
Proposition \ref{prop:Strong-Approximation} follows from Theorem
\ref{thm:Zorich-conj} together with strong approximation for finite
index subgroups of $\text{\ensuremath{\Sp}}_{2g}(\Z)$ . More precisely,
it is known that $\Theta_{\infty}(\Lambda_{\pi})$ contains the principal
congruence subgroup of $\Sp(H_{1}(\Z),\cap)$ of level 2, and that
this congruence subgroup projects onto $G(q)$ for all odd $q>1$
(see the Appendix, Proof of Corollary \ref{cor:gamma2}).

If we knew for $\sigma\neq0$ that $\Theta_{\infty}^{\sigma}(\Lambda_{\pi})$
generates a finite index subgroup of $G_{\sigma}(\infty)\cong H_{1}(\Z)\rtimes G(\infty)$
then the strong approximation hypothesis would follow readily. However,
this is not known at the current time for general components $\M$.
It is an amusing point of this paper that we can work around this
lack of knowledge, in particular, in the proof of Proposition \ref{prop:SSA}
below, we do not know which case occurs, but we win either way.
\end{rem}
Using Kazhdan's property (T), we can strengthen the strong approximation
statement of Proposition \ref{prop:Strong-Approximation} into the
uniform expansion of related Cayley graphs. This type of result has
been called \emph{super-strong} \emph{approximation} by some researchers;
see for example \cite{BO}. 
\begin{prop}[Super-strong approximation]
\label{prop:SSA}Let $\sigma\neq0$. Then \textbf{Uniform Expansion
}holds for the family of homomorphisms $\{\Theta_{q}^{\sigma}:\PMod(S_{\pi},\Sigma_{\pi})\to G_{\sigma}(q)\::\:q\text{ odd}\}$
and the Rauzy-Veech monoid $\Lambda_{\pi}$.
\end{prop}
\begin{proof}
Let $\GG_{\sigma}\leq G_{\sigma}(\infty$) be the group generated
by $\Theta_{\infty}^{\sigma}(\Lambda_{\pi})$. Then $\GG_{\sigma}$
fits into the short exact sequence
\[
1\to U_{\sigma}(\infty)\cap\GG_{\sigma}\to\GG_{\sigma}\to\GG\to1
\]
 where $\GG$ is the group generated by the action of $\Lambda_{\pi}$
on $H_{1}(\Z)$. By Theorem \ref{thm:Zorich-conj}, $\GG$ is a finite
index subgroup of $\Sp(H_{1}(\Z),\cap)$. We now distinguish two cases:
\begin{casenv}
\item $U_{\sigma}(\infty)\cap\GG_{\sigma}=\{1\}$. This implies $\GG_{\sigma}$
is isomorphic to $\GG$, a finite index subgroup of $\Sp(H_{1}(\Z),\cap)$.
\item $U_{\sigma}(\infty)\cap\GG_{\sigma}\neq\{1\}.$ Recalling from Lemma
\ref{lem:unipotent-identification} that $U_{\sigma}(\infty)$ is
naturally identified with $H_{1}(\Z)$, this implies that $U_{\sigma}(\infty)\cap\GG_{\sigma}$
is identified with a nonzero submodule of $H_{1}(\Z)$, moreover,
this submodule must be $\GG$-invariant. Let $1\leq r\leq2g$ denote
its rank. By strong approximation for $\GG$ (Corollary \ref{cor:gamma2}) we can
find a prime $p$ such that $\GG$ maps onto $\Sp(H_{1}(\F_{p}),\cap)$
and moreover, $U_{\sigma}(\infty)\cap\GG_{\sigma}$ maps to a dimension
$r$ $\Sp(H_{1}(\F_{p}),\cap)$-invariant vector subspace of $H_{1}(\F_{p})$.
But the only such invariant subspace, given $r\geq1$, is the whole
of $H_{1}(\F_{p})$. This fact follows for example since $\Sp(H_{1}(\F_{p}),\cap)$
acts transitively on nonzero vectors in $H_{1}(\F_{p})$ by Theorem
\ref{thm:transitive-on-unimodular}. Therefore $U_{\sigma}(\infty)\cap\GG_{\sigma}$
is identified with a rank $2g$ subgroup $A\leq H_{1}(\Z)$. In particular,
$A$ is finite index in $H_{1}(\Z)$. Therefore, $\GG_{\sigma}$ is
finite index in $G_{\sigma}(\infty)\cong H_{1}(\Z)\rtimes\Sp(H_{1}(\Z),\cap)$.
\end{casenv}
In either case there is a finite set $T\subset\Lambda_{\pi}$ such
that $\Theta_{\infty}^{\sigma}(T\cup T^{-1})$ generates $\GG_{\sigma}$.
We will establish \textbf{Uniform Expansion}\textbf{\emph{ }}using
this generating set. Let $\tilde{T}=\Theta_{\infty}^{\sigma}(T)\subset\GG_{\sigma}$.
It is sufficient to prove that if $f_{q}:\GG_{\sigma}\to G_{\sigma}(q)$
is the reduction modulo $q$ map, then the family of Cayley graphs
\begin{equation}
\{\:\Cayley(G_{\sigma}(q),f_{q}(\tilde{T}\cup\tilde{T}^{-1}))\:\}_{q\in\QQ}\label{eq:Cayley-family}
\end{equation}
 are uniform expanders. We write $\ell^{2}(G_{\sigma}(q))$ for functions
on $G_{\sigma}(q)$ with the standard Euclidean inner product, and
$\ell_{0}^{2}(G_{\sigma}(q))$ for the subspace of functions that
are orthogonal to the constant function. We obtain a unitary representation
$\rho_{q}$ of $\GG_{\sigma}$ on $\ell_{0}^{2}(G_{\sigma}(q))$,
and as is well-known, the uniform expansion of the family of graphs
in (\ref{eq:Cayley-family}) is equivalent to the existence of some
$\epsilon>0$, such that for all $q\in\QQ$, and all $v\in\ell_{0}^{2}(G_{\sigma}(q))$
with $\|v\|=1$, there is some element $t$ of $\tilde{T}$ such that
$\|\rho_{q}(t).v-v\|\geq\epsilon$. 

We will use Kazhdan's property (T) for $\GG_{\sigma}$ to deduce this,
but to do so, we must first note that there are no $\GG_{\sigma}$-invariant
vectors in $\ell_{0}^{2}(G_{\sigma}(q))$. This follows from Proposition
\ref{prop:Strong-Approximation}.

If Case 1 above occurs, then we are finished by property (T) for finite
index subgroups of $\Sp_{2g}(\Z)$ \cite{KAZHDAN}. If Case 2 occurs,
then we know that $H_{1}(\Z)\rtimes\Sp(H_{1}(\Z),\cap)$ has property
(T) by a result of Burger \cite[pg. 62, Example 2]{BURGER}. Hence
since $\GG_{\sigma}$ is finite index in $G_{\sigma}(\infty$), it
also has property (T). The result follows.
\end{proof}

\subsection{Establishing the inputs to the uniform spectral gap II: Quasirandomness
estimates.}\label{sec:bb2}

We continue to assume that $\sigma\neq0$. First we treat the specific
case that $q$ is a prime power.
\begin{lem}
\label{lem:pprime-power}There are constants $C>0$ and $D>0$ such
that if $q=p^{N}$ with p an odd prime and $N\geq1$ then any new irreducible
representation $(\rho, V)$ of $G_{\sigma}(p^{N})$ has dimension $\geq Cp^{ND}=Cq^{D}$.
\end{lem}
\begin{proof}
For $r\leq N$ we define the congruence subgroup of level $p^{r}$
in $G(p^{N})$, denoted by $G(p^{N},p^{r})$, to be the kernel of
the reduction mod $p^{r}$ map $G(p^{N})\to G(p^{r}).$ Similarly
we define $G_{\sigma}(p^{N},p^{r})$. Note $G(p^{N},p^{r})\leq G(p^{N},p^{r-1})$.
The statement that $\rho$ is new is equivalent to the statement that
$G_{\sigma}(p^{N},p^{r})$ is not in the kernel of $\rho$ for any
$0\leq r\leq N-1$, which is equivalent to $G_{\sigma}(p^{N},p^{N-1})$
not being in the kernel of $\rho$. 

Note that under the identification $G_{\sigma}(p^{N})\cong H_{1}(\Z_{p^{N}})\rtimes G(p^{N})$,
$G_{\sigma}(p^{N},p^{N-1})$ corresponds to $p^{N-1}H_{1}(\Z_{p^{N}})\rtimes G(p^{N},p^{N-1})$.
Therefore if $\rho$ is new then $\ker(\rho)$ cannot contain both
the congruence subgroups $G(p^{N},p^{N-1})$ and $p^{N-1}H_{1}(\Z_{p^{N}})$,
since these groups generate $G_{\sigma}(p^{N},p^{N-1})$. Hence at
least one of the following cases occurs:

\textbf{Case 1. $\ker(\rho)$ doesn't contain $p^{N-1}H_{1}(\Z_{p^{N}})$.}
In this case we decompose $V$ into subspaces $V_{\chi}$ where $H_{1}(\Z_{p^{N}})\leq G_{\sigma}(p^{N})$
acts by the character $\chi$:

\[
V=\bigoplus_{\chi\in\widehat{H_{1}(\Z_{p^{N}})}}V_{\chi}.
\]
By using the nondegenerate symplectic form $\cap$ on $H_{1}(\Z_{p^{N}})$
we obtain an isomorphism
\begin{equation}
H_{1}(\Z_{p^{N}})\to\widehat{H_{1}(\Z_{p^{N}})},\quad v\mapsto\exp\left(\frac{2\pi i\:v\cap\bullet}{p^{N}}\right).\label{eq:duality-identificiation}
\end{equation}
The conjugation action of $G_{\sigma}(p^{N})$ on the normal abelian
subgroup $U_{\sigma}(p^{N})$, under the identification $U_{\sigma}(q)\cong H_{1}(\Z_{p^{N}})$,
factors through the action of $\Sp(H_{1}(\Z_{p^{N}}),\cap$) on $H_{1}(\Z_{p^{N}})$.
As such, if $V_{\chi}\neq\emptyset$ then $V_{\chi'}\neq\emptyset$
for all $\chi'=\chi\circ g$ with $g\in\Sp(H_{1}(\Z_{p^{N}}),\cap)$
and the identification (\ref{eq:duality-identificiation}) identifies
the action of $\Sp(H_{1}(\Z_{p^{N}}),\cap)$ on $\chi\in\widehat{H_{1}(\Z_{p^{N}})}$
with the action on $H_{1}(\Z_{p^{N}}).$ The upshot of this discussion
is that if we can find a character $\chi$ of $H_{1}(\Z_{p^{N}})$
with $V_{\chi}\neq\emptyset$, corresponding to $v\in H_{1}(\Z_{p^{N}})$
under (\ref{eq:duality-identificiation}), then $\dim V$ is at least
the size of the orbit $\Sp(H_{1}(\Z_{p^{N}}),\cap).v$.

The assumption that $\ker(\rho)$ doesn't contain $p^{N-1}H_{1}(\Z_{p^{N}})$
exactly translates to the existence of a character $\chi$ of $H_{1}(\Z_{p^{N}})$
with $V_{\chi}\neq0$ corresponding via (\ref{eq:duality-identificiation})
to $v\in H_{1}(\Z_{p^{N}})$ where $v$ is unimodular. The number
of unimodular vectors in $H_{1}(\Z_{p^{N}})$ is $p^{2gN}\left(1-\frac{1}{p^{2g}}\right)$
(see Lemma \ref{lem;num-unimod}) so this implies by Theorem \ref{thm:transitive-on-unimodular}
\[
\dim(V)\geq|G(p^{N}).v|=p^{2gN}\left(1-\frac{1}{p^{2g}}\right)\geq\frac{p^{2gN}}{2}.
\]
\textbf{Case 2. $\ker(\rho)$ doesn't contain $G(p^{N},p^{N-1})$.
}Then we can restrict $\rho$ to the $G(p^{N})$ factor of $G_{\sigma}(p^{N})$
to obtain $\rho'$ that contains a new irreducible representation
of $G(p^{N})$. Hence from \cite[Prop. 5.1]{MAGEEKZ}, for some $C>0$
and $D>0$, $\dim V\geq Cp^{ND}$.
\end{proof}
\begin{proof}[Proof of {\bf Property II: Quasirandomness}]
We reduce to the case that $q$ is a prime power using the Chinese
remainder theorem. Let $q=\prod_{p}p^{e_{p}}$ where $p$ are distinct
odd primes. The Chinese remainder theorem gives an isomorphism
\[
G_{\sigma}(\Z_{q})\cong\prod_{p}G_{\sigma}(\Z_{p^{e_{p}}}).
\]
This implies that $\rho$ splits as a tensor product
\[
V=\otimes W_{p}
\]
 where $W_{p}$ is an irreducible representation of $G_{\sigma}(\Z_{p^{e_{p}}})$.
Also, since $\rho$ is new, each $W_{p}$ must be a new representation
of $G_{\sigma}(\Z_{p^{e_{p}}})$. Therefore using Lemma \ref{lem:pprime-power}
we get
\[
\dim(\rho)\geq\prod_{p}Cp^{e_{p}D}\geq C^{\omega(q)}q^{D}
\]
where $\omega(q)$ denotes the number of distinct prime factors of
$q$. We have $\omega(q)=O(\frac{\log q}{\log\log q})$ so we get
for some $c>0$, 
\[
\dim(\rho)\geq q^{D-\frac{c}{\log\log q}}.
\]
So for $q$ sufficiently large (bigger than a constant depending on
$g$), we have $\dim(\rho)\geq q^{D/2}$. The result follows.
\end{proof}

\section{Counting saddle connections}

\subsection{Holonomy and configurations of saddle connections}\label{holo}

We will denote by $\tilde{S}$ a translation surface structure on
$(S,\Sigma)$ whose isotopy class lies in $\X_{\kappa}$. A \emph{saddle
connection }on $\tilde{S}$ is a flat directed geodesic in $S-\Sigma$
with end points in $\Sigma$. A saddle connection is \emph{closed}
if its endpoints coincide. Following Eskin-Masur-Zorich \cite{EMZ},
we define a \emph{configuration of saddle connections} on $(S,\Sigma)$.
There are two types of configurations depending on whether the saddle
connections of that configuration will join two distinct zeros, or
be closed.

\textbf{A configuration $\CC$ of saddle connections joining distinct
zeros} consists of the following data:
\begin{enumerate}
\item A pair of points\footnote{Here we refine the Eskin-Masur-Zorich notion of configuration by specifying
the points, and not simply the values of $\kappa$ at these points.} $z_{1},z_{2}\in\Sigma$.
\item A number $p$ describing how many saddle connections we consider,
and a list of angle parameters $a'_{i},a''_{i}$, $1\leq i\leq p$.
\end{enumerate}
Given the above data, we say that an ordered tuple $(\gamma_{1},\ldots,\gamma_{p})$
of saddle connections on $\tilde{S}$ is \emph{of type $\CC$} if
the $\gamma_{i}$ are disjoint, each join $z_{1}$ to $z_{2}$, directed
from $z_{1}$ to $z_{2}$, are homologous, the cyclic order (clockwise)
that the $\gamma_{i}$ emanate from $z_{1}$ agrees with the ordering
of the tuple, and the angle between $\gamma_{i}$ and $\gamma_{i+1}$
at $z_{1}$ is $2\pi(a'_{i}+1)$, and the angle between $\gamma_{i}$
and $\gamma_{i+1}$ at $z_{2}$ is $2\pi(a''_{i}+1)$.

\textbf{A configuration $\CC$ of closed saddle connections} consists
of the following data:
\begin{enumerate}
\item A number $p$ specifying the number of saddle connections and a subset
$J\subset\{1,\ldots,p\}$.
\item Angle parameters $a'_{i},a''_{i},b'_{i},b''_{i}$ that describe the
angles between incident saddle connections.
\item A fixed function $Z:\{1,\ldots,p\}\to\Sigma$ that specifies the endpoints
of the saddle connections. Here again we refine the Eskin-Masur-Zorich
definition by specifying the endpoints of the saddle connections.
\end{enumerate}
Now an ordered tuple of saddle connections $(\gamma_{1},\ldots,\gamma_{p})$
is of type $\CC$ if they are disjoint, closed, homologous to one
another in absolute homology, the endpoints of $\gamma_{i}$ are at
$Z(i)\in\Sigma$ and further conditions are met depending on $J$
and the angle parameters; since these are quite involved and do not
add much to this paper we refer the reader to \cite[\S 11.1]{EMZ}
for the conditions coming from $J$ and $a'_{i},a''_{i},b'_{i},b''_{i}$. 

Any saddle connection $\gamma$ on a flat surface $\tilde{S}$ gives
rise to a holonomy vector $\hol(\gamma)$ as follows. For this, it
is convenient to note that a translation surface structure $\tilde{S}$
on $(S,\Sigma)$ corresponds uniquely to a holomorphic one form $\omega$
with respect to some complex structure and with zeros of $\omega$
given by $\Sigma$. For more details on this correspondence we refer
the reader to \cite{EMZ}. 

Then if the endpoints of $\gamma$ are at $z_{1},z_{2}$ we define
\begin{equation}
\hol_{\tilde{S}}(\gamma)\stackrel{\mathrm{def}}{=}\int_{z_{1}}^{z_{2}}\omega\in\C=\R^{2}.\label{eq:hol-defn}
\end{equation}
Intuitively, the holonomy vector tells us how we are translated as
we move from $z_{1}$ to $z_{2}$ along $\gamma$. Note that the definition
(\ref{eq:hol-defn}) implies that if $\gamma_{1}$ and $\gamma_{2}$
are homologous then $\hol_{\tilde{S}}(\gamma_{1})=\hol_{\tilde{S}}(\gamma_{2})$.
Therefore we obtain for any configuration $\CC$ a mapping 
\[
\hol_{\tilde{S}}:\{\text{tuples of saddle connections of type \ensuremath{\CC} on \ensuremath{\tilde{S}}\}\ensuremath{\to}}\R^{2}.
\]

For a fixed configuration $\CC$, for each $x\in\M$ with representative
translation surface $\tilde{S}$ we denote by $V_{\CC}(x)\subset\R^{2}$
the holonomy vectors arising via the map $\hol_{\tilde{S}}$ through
the tuples of saddle connections of configuration $\CC$. This collection
$V_{\CC}(x)$ depends only on $x\in\M$.

Let $\mu$ be the Masur-Veech measure on $\M$. It is proved by Eskin-Masur-Zorich
\cite[Prop. 3.1]{EMZ} that for $\mu$-a.e. $x\in\M$, the homology
class of a saddle connection is uniquely determined by its holonomy
vector. Henceforth we always assume $x$ has this property.

Given a configuration $\CC$, the endpoints $z_{1},z_{2}$ specify
a class $\sigma=\sigma(\CC)=z_{1}-z_{2}\in\tilde{H}_{0}(\Sigma,\Z)$.
This class is trivial precisely when the configuration consists
of closed saddle connections. We view $\CC$ and hence $\sigma$ as
fixed henceforth.

Note that if $\overrightarrow{\gamma}=\{\gamma_{1},\ldots,\gamma_{p}\}$
is a tuple of saddle connections of type $\CC$, since the $\gamma_{i}$
are homologous, the images of the $\gamma_{i}$ in $H_{1}^{\rel}(\Z_{q})$
are the same for each $i$, and moreover, lie in $H_{1}^{\sigma}(\Z_{q})$
where $\sigma=\sigma(\CC)$. We write $\pi_{q}(\overrightarrow{\gamma})\in H_{1}^{\sigma}(\Z_{q})$
for this class. 

Recalling the covering space $\M(\Theta_{q}^{\sigma})$ from Section
\ref{subsec:Connected-components-of}, given any $\xi\in H_{1}(S,\Sigma,\Z_{q})$,
for $x\in\M(\Theta_{q}^{\sigma})$ we let $V_{\CC,\xi}(x)$ denote
the holonomy vectors arising from any $\overrightarrow{\gamma}$ of
configuration $\CC$ with $\pi_{q}(\overrightarrow{\gamma})=\xi$.
Note that the function $V_{\CC,\xi}(x)$ is well defined on $\M(\Theta_{q}^{\sigma})$,
since points in $\M(\Theta_{q}^{\sigma})$ are isotopy classes of
translation surfaces, up to mapping classes that act trivially on
$H_{1}^{\sigma}(\Z_{q})$.

We define the counting function
\[
N^{\CC}(x;L,\xi)=|V_{\CC,\xi}(x)\cap B(L)|,\quad x\in\M(\Theta_{q}^{\sigma}),\:L\in\R_{+},\:\xi\in H_{1}^{\sigma}(\Z_{q}).
\]
The goal of this section is to use the uniform spectral gap (Theorem
\ref{thm:Uniform-spectral-gap}) to obtain uniform estimates for the
quantity $N^{\CC}(x;L,\xi)$. We will prove the following refinement of Theorem \ref{thm:main-theorem}.

\begin{thm}\label{thm:tech}
Let $\M$ be a connected component of $\H(\kappa)$ and $\CC$ a configuration of saddle connections. There are constants $Q_0\in\Z_+$, $\eta,\alpha>0$ depending on $\M$, and  $c\geq 0$ depending on  $\CC$ and $\M$, such that the following hold. For $\mu$-almost
all $x\in\M$,  for
all $q$ coprime to $Q_0$, for all $\xi\in H_{1}(S,\Sigma,\Z_{q})$ such that $\delta(\xi) = \sigma(\CC)$, we have
\[
N^\CC(x;L,\xi)=\frac{\pi c L^{2}}{|\PMod(S,\Sigma).\xi|}+O(q^{\alpha}L^{2-\eta}).
\]
The implied constant depends on $x$. 
\end{thm}

Theorem \ref{thm:tech} implies Theorem \ref{thm:main-theorem} as the $c=c(\M,\CC)$ in Theorem \ref{thm:tech} is the \emph{Siegel-Veech constant} of $\M$ and $\CC$ in the sense of Eskin-Masur-Zorich \cite{EMZ}. Eskin, Masur, and Zorich prove that given $\M$, there are only finitely many $\CC$ such that $c(\M,\CC)\neq 0$. Theorem  \ref{thm:main-theorem} follows by summing up Theorem \ref{thm:tech} over configurations of saddle connections with $p=1$, i.e., only one saddle connection in the configuration, and $\sigma(\CC) =\delta(\xi)$.

\subsection{The Siegel-Veech formula}\label{sec-SV}

We write $B(R)$ for the closed Euclidean ball of radius $R$ centered
at $(0,0)$ in $\R^{2}$. For $\psi$ continuous and compactly supported
on $\R^{2}$, the \emph{Siegel-Veech transform} of $\psi$, relative
to the fixed configuration $\CC$, is a function $\hat{\psi}:\M\to\R$
defined by 
\[
\hat{\psi}(x)\stackrel{\mathrm{def}}{=}\sum_{v\in V_{\CC}(x)}\psi(v).
\]
Veech proved in \cite{VEECHSIEGEL} (see also \cite{EM}) that if
$\psi\in C_{c}(\R^{2})$ then $\hat{\psi}\in L^{1}(\mu)$; it has
recently been shown by Athreya, Cheung and Masur \cite{ACM} that
\begin{thm}[Athreya-Cheung-Masur]
\label{thm:SVin L2}If $\psi\in C_{c}(\R^{2})$, then $\hat{\psi}\in L^{2}(\mu)$.
\end{thm}
The \emph{Siegel-Veech formula} proved by Veech \cite{VEECH} states
that
\begin{equation}
\int_{\M}\hat{\psi}(\nu)d\mu=c(\CC)\int_{\R^{2}}\psi d\Leb\label{eq:SV}
\end{equation}
where $c(\CC)$ is a positive \emph{Siegel-Veech constant} depending
only on $\CC$ (and $\mu)$. We need a version of the formula (\ref{eq:SV})
that takes into account the parameter $\xi\in H_{1}^{\sigma}(\Z_{q})$.
Therefore, we define a relative version of the Siegel-Veech transform:
for any $\psi\in C_{c}(\R^{2})$ we define 
\begin{align*}
{\widehat{\psi}}^{\xi}(x) & \stackrel{\mathrm{def}}{=}\sum_{v\in V_{\CC,\xi}(x)}\psi(v).
\end{align*}

\begin{thm}
\label{thm:siegelveechformula} Suppose $q$ is odd. For $\psi\in C_{c}(\R^{2},dx)$
we have ${\widehat{\psi}}^{\xi}\in L^{1}(\mu_{q}^{\sigma})$ and 
\begin{equation}
\mu_{q}^{\sigma}({\widehat{\psi}}^{\xi})=|\Stab_{G_{\sigma}(q)}(\xi)|c(\CC)\int_{\R^{2}}\psi\;d\Leb.\label{eq:SVcongruence}
\end{equation}
\end{thm}
\begin{proof}
First note that for $x\in\M(\Theta_{q}^{\sigma})$,

\[
|\hat{\psi}^{\xi}(x)|\leq\sum_{v\in V_{\CC,\xi}(x)}|\psi(v)|\leq\sum_{v\in V_{\CC}(x)}|\psi(v)|=\cover_{\Theta_{q}^{\sigma}}^{*}\widehat{|\psi|}(x).
\]
Since $|\psi|\in C_{c}(\R^{2})$, therefore $\widehat{|\psi|}\in L^{1}(\M,\mu)$,
and $\cover_{\Theta_{q}^{\sigma}}$ has bounded fibers for fixed $q$,
we obtain that $\hat{\psi}^{\xi}\in L^{1}(\M(\Theta_{q}^{\sigma},\mu_{q}^{\sigma}))$.
To obtain (\ref{eq:SVcongruence}), we `unfold'. Let $\mathcal{F}$
be a fundamental domain for the action of $G_{\sigma}(q)$ on $\M(\Theta_{q}^{\sigma})$.
Then

\begin{align*}
\int_{\M(\Theta_{q}^{\sigma})}{\widehat{\psi}}^{\xi}(x)d\mu_{q}^{\sigma} & =\int_{\mathcal{F}}\sum_{\gamma\in G_{\sigma}(q)}{\widehat{\psi}}^{\xi}(\gamma x)d\mu(x)=\int_{\mathcal{F}}\sum_{\gamma\in G_{\sigma}(q)}\sum_{v\in V_{\CC,\xi}(\gamma x)}\psi(v)d\mu(x).
\end{align*}
Now we use the identity $V_{\CC,\xi}(\gamma x)=V_{\CC,\gamma^{-1}\xi}(x)$
to obtain

\begin{align*}
 & =\int_{\mathcal{F}}\sum_{\gamma\in G_{\sigma}(q)}\sum_{v\in V_{\CC,\gamma^{-1}\xi}(x)}\psi(v)d\mu(x)=|\Stab_{G_{\sigma}(q)}(\xi)|\int_{\mathcal{F}}\sum_{v\in V_{\CC}(x)}\psi(v)d\mu(x)\\
 & =|\Stab_{G_{\sigma}(q)}(\xi)|\int_{\M}\hat{\psi}(x)d\mu(x)=|\Stab_{G_{\sigma}(q)}(\xi)|c(\CC)\int_{\R^{2}}\psi(x)d\Leb.
\end{align*}
For the final equality above
we used the Siegel-Veech formula (\ref{eq:SV}) for $\M$.
To obtain the second equality of the first line, we used that $G_{\sigma}(q)$ acts transitively
on the possible $\xi'\in H_{1}^{\sigma}(\Z_{q})$ that can arise from
saddle connections of type $\CC$ (Theorem \ref{thm:transitive-on-unimodular}
and Lemma \ref{lem:transitive-relative}). For example, if $q$ is prime, and $\CC$ is a configuration 
of closed saddle connections, then any non-zero element of $H_1(\Z_q)=H^\sigma_1(\Z_q)$ can arise
from a closed saddle connection and $G_\sigma(q)=G(q)$ acts transitively on $H_1(\Z_q)-\{0\}$.
\end{proof}
We now give estimates for Sobolev norms of $\hat{\psi}^{\xi}$ that
we will need for application of the spectral gap result. Let
\[
\omega=\left(\begin{array}{cc}
0 & -1\\
1 & 0
\end{array}\right).
\]
Since $\M(\Theta_{q}^{\sigma})$ carries an action of $\SL_{2}(\R)$,
we may consider the space $S_{K}(\M(\Theta_{q}^{\sigma}))$ of functions
$f$ in $L^{2}(\mu_{q}^{\sigma})$ for which the following limit exists
in the $L^{2}(\mu_{q}^{\sigma})$ norm:
\[
\omega.f\stackrel{\mathrm{def}}{=}\lim_{t\to0}\frac{1}{t}\left(\exp(t\omega)f-f\right).
\]
We equip the space of functions $S_{K}(\M(\Theta_{q}^{\sigma}))$
with the norm
\[
S_{K}(f)^{2}\stackrel{\mathrm{def}}{=}\|f\|_{L^{2}(\mu_{q}^{\sigma})}^{2}+\|\omega.f\|_{L^{2}(\mu_{q}^{\sigma})}^{2}.
\]
We will write $\partial_{\theta}$ for the partial derivative in $\R^{2}$
with respect to the angle parameter $\theta$ in polar coordinates
centered at $(0,0)$.
\begin{lem}
\label{lem:sob-estimate-SV}For $\psi\in C_{c}^{1}(\R^{2})$ with
$\supp(\psi)\subset B(r)$ we have $\hat{\psi}^{\xi}\in S_{K}(\M(\Theta_{q}^{\sigma}))$
and 
\begin{equation}
S_{K}(\hat{\psi}^{\xi})^{2}\ll_{r}|G_{\sigma}(q)|\left(\|\psi\|_{\infty}^{2}+\|\partial_{\theta}\psi\|_{\infty}^{2}\right).\label{eq:sobres}
\end{equation}
\end{lem}
\begin{proof}
Firstly, as in the proof of Theorem \ref{thm:siegelveechformula},
we have $|\hat{\psi}^{\xi}(x)|\leq\cover_{\Theta_{q}^{\sigma}}^{*}\widehat{|\psi|}(x)$
and since the fibers of $\cover_{\Theta_{q}^{\sigma}}$ are bounded
in size by $|G_{\sigma}(q)|$ we deduce since $\widehat{|\psi|}$
is in $L^{2}(\mu)$ by Theorem \ref{thm:SVin L2}, we have $\hat{\psi}^{\xi}\in L^{2}(\mu_{q}^{\sigma})$
and
\begin{equation}
\|\hat{\psi}^{\xi}\|_{L^{2}(\mu_{q}^{\sigma})}^{2}\leq|G_{\sigma}(q)|\|\widehat{|\psi|}\|_{L^{2}(\mu)}^{2}\ll_{r}|G_{\sigma}(q)|\|\psi\|_{\infty}^{2}.\label{eq:sob1}
\end{equation}
To obtain the last inequality we used that since $|\psi|\leq\|\psi\|_{\infty}\mathbf{1}_{B(r)}$,
$0\leq\widehat{|\psi|}\leq\|\psi\|_{\infty}\widehat{\mathbf{1}_{B(r)}}$.

Noting that $\omega.\hat{\psi}^{\xi}=\widehat{\partial_{\theta}\psi}^{\xi}$
(since the Siegel-Veech transform is $K$-equivariant, and using the
compact support of $\psi$ to interchange sums and derivatives, cf.
\cite[proof of Lemma 4.1]{NRW}), and repeating exactly the same argument
as before with $\psi$ replaced by $\partial_{\theta}\psi$, we obtain
that $\omega.\hat{\psi}^{\xi}\in L^{2}(\mu_{q}^{\sigma})$ and 
\begin{equation}
\|\omega.\hat{\psi}^{\xi}\|_{L^{2}(\mu_{q}^{\sigma})}^{2}\ll_{r}|G_{\sigma}(q)|\|\partial_{\theta}\psi\|_{\infty}^{2}.\label{eq:sob2}
\end{equation}
Therefore we have established that $\hat{\psi}^{\xi}\in S_{K}(\M(\Theta_{q}^{\sigma}))$
and adding (\ref{eq:sob1}) and (\ref{eq:sob2}) together gives (\ref{eq:sobres}).
\end{proof}

\subsection{A pointwise ergodic theorem}

In this section we will explain how the uniform spectral gap from
Theorem \ref{thm:Uniform-spectral-gap} leads to estimates for the
circle averaging operators
\[
\Ave_{K,t}[f](x)\stackrel{\mathrm{def}}{=}\int_{K}f(a_{t}k.x)dm_{K}(k)
\]
where $K=SO(2)$, $m_{K}$ is the probability Haar measure on $K$,
$a_{t}=\left(\begin{array}{cc}
e^{t} & 0\\
0 & e^{-t}
\end{array}\right)$, and the operator $\Ave_{K,t}$ acts on the Sobolev space $S_{K}(\M(\Theta_{q}^{\sigma}))$.
The following theorem is a $q$-uniform version of \cite[Theorem 3.3]{NRW}
that follows arguments of Veech \cite[Ch. 14]{VEECHSIEGEL} and Eskin-Margulis-Mozes
\cite[eq. 3.32]{EMM}.
\begin{thm}
\label{thm:L2 norm of averaging op}Suppose that $q$ is coprime to the constant $Q_{0}$ from
Theorem \ref{thm:Uniform-spectral-gap}. There is $\kappa>0$ and
$c>0$, both independent of $q$, such that for any $f\in S_{K}(\M(\Theta_{q}^{\sigma}))$
for which $\int fd\mu_{q}^{\sigma}=0$, for all $t>0$
\[
\|\Ave_{K,t}f\|_{2}^{2}\leq ce^{-t\kappa}S_{K}(f)^{2}.
\]
\end{thm}
\begin{proof}
For fixed $q$ coprime to $Q_{0}$, so that the $\SL_{2}(\R)$ action on $\M_{q}^{\sigma}$
is ergodic, this follows directly from \cite[Theorem 3.3]{NRW}. The
referenced theorem applies to probability measures, but both sides
of the inequality behave in the same way when scaling the measure.
It remains to check that $c$ and $\kappa$ can be chosen independent
of $q$. The constant $\kappa$ depends only on the spectral gap so
can be taken to be uniform by Theorem \ref{thm:Uniform-spectral-gap}.
Inspection of the proof of \cite[Theorem 3.3]{NRW} shows that $c$
also only depends on the spectral gap, so is also uniform.
\end{proof}
We now use the Borel-Cantelli lemma, in a variant of the argument
from \cite[Theorem 3.5]{NRW}, to obtain a pointwise ergodic theorem
for Siegel-Veech transforms.
\begin{thm}[Pointwise ergodic theorem for Siegel-Veech transforms]
\label{thm:pointwise-ergodic}There exist constants $\kappa',\alpha>0$
such that the following hold. Let $Q_{0}$ be the constant from Theorem
\ref{thm:Uniform-spectral-gap}. For any $r>0$, for any sequence
$\{\psi_{n}\}_{n\in\N}$ with $\psi_{n}\in C_{c}^{1}(\R^{2})$ and
$\supp(\psi_{n})\subset B(r)$, for any sequence of real numbers $\{t_{n}\}_{n\in\N}$
with $t_{n}>0$ for all $n$ and $\sum e^{-t_{n}\kappa'}<\infty$,
there exists a set of full measure $\M'\subset\M$ such that for every
$x'\in\M'$, there exists $n_{x'}\geq0$ such that for all $q$ coprime to $Q_0$,
for all $n\geq n_{x'}$, for all $x\in\cover_{\Theta_{q}^{\sigma}}^{-1}(x')$,
\[
|\Ave_{K,t_{n}}[{\widehat{\psi_{n}}}^{\xi}](x)-|G_{\sigma}(q)|^{-1}\mu_{q}^{\sigma}({\widehat{\psi_{n}}}^{\xi})|\leq e^{-\kappa't_{n}}q^{\alpha}\sqrt{\|\psi_{n}\|_{\infty}^{2}+\|\partial_{\theta}\psi_{n}\|_{\infty}^{2}}.
\]
\end{thm}
\begin{proof}
Let $f_{n,q}\stackrel{\mathrm{def}}{=}{\widehat{\psi_{n}}}^{\xi}-|G_{\sigma}(q)|^{-1}\mu_{q}^{\sigma}({\widehat{\psi_{n}}}^{\xi})$.
Noting that $\mu_{q}^{\sigma}(f_{n,q})=0$ we have $\|f_{n,q}\|_{L^{2}(\mu_{q}^{\sigma})}\leq\|{\widehat{\psi_{n}}}^{\xi}\|_{L^{2}(\mu_{q}^{\sigma})}$.
Also noting $\omega.f_{n,q}=\omega.{\widehat{\psi_{n}}}^{\xi}$ we
have
\[
S_{K}^{2}(f_{n,q})\leq S_{K}^{2}({\widehat{\psi_{n}}}^{\xi}).
\]
Thus by Lemma \ref{lem:sob-estimate-SV} there is a constant $c=c(r)>0$
such that
\begin{equation}
S(f_{n,q})^{2}\leq c|G_{\sigma}(q)|\left(\|\psi_{n}\|_{\infty}^{2}+\|\partial_{\theta}\psi_{n}\|_{\infty}^{2}\right).\label{eq:sobfnq}
\end{equation}
Now we set up to use the Borel-Cantelli lemma. We define for some
$\alpha,\kappa'>0$ to be chosen

\[
f_{n}^{(Q)}(x')\stackrel{\mathrm{def}}{=}\sum_{q \: : \: (q,Q_0)=1 , \:q\leq Q}q^{-\alpha}\sum_{\gamma\in G_{\sigma}(q)}|\Ave_{K,t_{n}}[f_{n,q}](\gamma x)|,\quad f_{n}^{(Q)}:\M\to\R_{\geq0},
\]
where $x$ is an arbitrary lift of $x'$ to $\M_{q}^{\sigma}$. Let
\[
U_{n,Q}\stackrel{\mathrm{def}}{=}\left\{ x'\in\M\::\:f_{n}^{(Q)}(x')\geq e^{-\kappa't_{n}}\sqrt{\|\psi_{n}\|_{\infty}^{2}+\|\partial_{\theta}\psi_{n}\|_{\infty}^{2}}\right\} .
\]
Then $U_{n,Q}$ is a $\mu$-measurable subset of $\M$ and 
\begin{equation}
\mu(U_{n,Q})\leq\left(e^{-\kappa't_{n}}\sqrt{\|\psi_{n}\|_{\infty}^{2}+\|\partial_{\theta}\psi_{n}\|_{\infty}^{2}}\right)^{-1}\int_{U_{n,Q}}f_{n}^{(Q)}d\mu.\label{eq:BC1}
\end{equation}
Furthermore, 
\begin{equation}
\int_{U_{n,Q}}f_{n}^{(Q)}d\mu\leq\int_{\M}f_{n}^{(Q)}d\mu=\sum_{Q\geq q\geq q_{0},q\text{ odd}}q^{-\alpha}\int|\Ave_{K,t_{n}}[f_{n,q}]|d\mu_{q}^{\sigma}\label{eq:BC2}
\end{equation}
by unfolding. Now using Cauchy-Schwarz and Theorem \ref{thm:L2 norm of averaging op}
together with (\ref{eq:sobfnq}) we get 
\begin{equation}
\int|\Ave_{K,t_{n}}[f_{n,q}]|d\mu_{q}^{\sigma}\leq|G_{\sigma}(q)|c^{1/2}e^{-t_{n}\kappa/2}\sqrt{\|\psi\|_{\infty}^{2}+\|\partial_{\theta}\psi\|_{\infty}^{2}}\label{eq:BC3}
\end{equation}
where $\kappa$ is the constant given by Theorem \ref{thm:L2 norm of averaging op}.
Putting (\ref{eq:BC1}), (\ref{eq:BC2}), and (\ref{eq:BC3}) together
we obtain
\[
\mu(U_{n,Q})\leq c^{1/2}\sum_{q \: : \: (q,Q_0)=1 , \:q\leq Q}q^{-\alpha}|G_{\sigma}(q)|e^{(\kappa'-\kappa/2)t_{n}}.
\]
Now define $U_{n}=\bigcup_{Q\geq q_{0}}U_{n,Q}$. Note that $U_{n,Q}\subset U_{n,Q+1}$.
We have that $U_{n}$ is $\mu$-measurable. Note for for some $D=D(\M)$,
$|G_{\sigma}(q)|\ll q^{D}$ so \textbf{we now choose }$\alpha=D+2$
so that $\mu(U_{n,Q})\leq c'e^{(\kappa'-\kappa/2)t_{n}}$ for all
$Q$. Therefore $\mu(U_{n})\leq c'e^{(\kappa'-\kappa/2)t_{n}}$. \textbf{We
now choose $\kappa'=\kappa/4$ }so $\mu(U_{n})\leq c'e^{-\kappa t_{n}/4}$
and hence
\[
\sum_{n}\mu(U_{n})\ll_{\M,r}\sum_{n}e^{-\kappa't_{n}/4}<\infty
\]
is finite by hypothesis. Therefore by the Borel-Cantelli lemma there
exists $\M'\subset\M$ with $\mu(\M')=1$, so that for any $x'\in\M'$
there exists $n_{x'}$ such that for any $n\geq n_{x'}$, $x'\notin U_{n}$.
Finally, the statement $x\notin U_{n}$ implies that for any $x\in\cover_{\Theta_{q}^{\sigma}}^{-1}(x')$,
\begin{align*}
|\Ave_{K,t_{n}}[{\widehat{\psi_{n}}}^{\xi}](x)-|G_{\sigma}(q)|^{-1}\mu_{q}^{\sigma}({\widehat{\psi_{n}}}^{\xi})| & =|\Ave_{K,t_{n}}[f_{n,q}](x)|\\
 & \leq q^{\alpha}e^{-\kappa't_{n}}\sqrt{\|\psi_{n}\|_{\infty}^{2}+\|\partial_{\theta}\psi_{n}\|_{\infty}^{2}}
\end{align*}
as required.
\end{proof}

\subsection{The counting argument: proof of Theorem \ref{thm:tech}}\label{proof}

\begin{proof}[Proof of Theorem \ref{thm:tech}]
Here we follow the argument of Nevo-R\"{u}hr-Weiss \cite[\S 5]{NRW},
that is a refinement of arguments in \cite[Lemma 3.6]{EMM} and \cite[Lemma 3.4]{ESKINMASUR}.
We are able to avoid the arguments of \cite[\S 4]{NRW} about `cutting
off the cusp' due to our use of the result of Athreya-Cheung-Masur,
Theorem \ref{thm:SVin L2}. Here we will not try to optimize constants.

Let $\delta>0$ be a small parameter that will control smoothing of
indicator functions in the following argument. We let $\theta=\delta^{1/2}$.

Let $W_{1}$ be the triangle in $\R^{2}$ with vertices at $(0,0)$,
$(\sin\theta,\cos\theta)$, $(-\sin\theta,\cos\theta)$. Let $W_{2}$
be the triangle with vertices $(0,0)$, $(\tan\theta,1)$, $(-\tan\theta,1)$.
Then $W_{1}$ and $W_{2}$ are similar with the same apex angle $2\theta$
at $(0,0)$ and $W_{1}\subset W_{2}$.

Now $t>1$ will be a parameter related to the counting parameter.
The diagonal element $a_{-t}$ maps $W_{1}$ and $W_{2}$ to thinner
triangles with apex angle given by $2\theta_{t}$ such that
\[
\tan\theta_{t}=e^{-2t}\tan\theta.
\]
The same arguments as in \cite[pg. 11]{NRW} give
\begin{equation}
\Ave_{K,t}[\widehat{\mathbf{1}_{W_{1}(\theta)}}^{\xi}](x)\leq\frac{\theta_{t}}{\pi}N^{\CC}(x;e^{t},\xi)\leq\Ave_{K,t}[\widehat{\mathbf{1}_{W_{2}(\theta)}}^{\xi}](x).\label{eq:sandwich1}
\end{equation}
Nevo, R\"{u}hr, and Weiss prove \cite[pg. 12]{NRW} there exist continuously
differentiable functions $\psi_{(\pm,\delta)}$ on $\R^{2}$ such
that 
\begin{equation}
\mathbf{\psi_{(-,\delta)}\leq}\mathbf{1}_{W_{1}(\theta)}\leq\mathbf{1}_{W_{2}(\theta)}\leq\psi_{(+,\delta)},\label{eq:pointwise}
\end{equation}
\begin{equation}
\|\psi_{(\pm,\delta)}\|_{\infty}\leq1,\quad\|\partial_{\theta}\psi_{(\pm,\delta)}\|_{\infty}\ll\delta^{-1},\label{eq:banach-bound-for-smoothed-functions}
\end{equation}
\[
\supp(\psi_{(\pm,\delta)})\subset B(1),
\]
and
\begin{equation}
\int_{\R^{2}}\psi_{(\pm,\delta)}d\Leb=e^{2t}\theta_{t}+O(\delta^{1/2}e^{2t}\theta_{t}).\label{eq:integral-of-approximant}
\end{equation}
The result of (\ref{eq:pointwise}) together with (\ref{eq:sandwich1})
is that

\[
\frac{\pi}{\theta_{t}}\Ave_{K,t}[\widehat{\psi_{(-,\delta)}}^{\xi}](x)\leq N^{\CC}(x;e^{t},\xi)\leq\frac{\pi}{\theta_{t}}\Ave_{K,t}[\widehat{\psi_{(+,\delta)}}^{\xi}](x).
\]
We now choose a sequence of $\{\delta_{n}\}_{n\in\N}$ (hence $\{\theta_{n}\}_{n\in\N})$
as well as a sequence of $\{t_{n}\}_{n\in\N}$. Repeating the previous
arguments for each $\theta_{n}$ we obtain sequences $\{\psi_{(+,\delta_{n})}^{n}\}_{n\in\N}$
and $\{\psi_{(-,\delta_{n})}^{n}\}_{n\in\N}$ such that

\begin{equation}
\frac{\pi}{\theta_{t_{n}}}\Ave_{K,t_{n}}[\widehat{\psi_{(-,\delta_{n})}^{n}}^{\xi}](x)\leq N^{\CC}(x;e^{t_{n}},\xi)\leq\frac{\pi}{\theta_{t_{n}}}\Ave_{K,t_{n}}[\widehat{\psi_{(+,\delta_{n})}^{n}}^{\xi}](x)\label{eq:sandwich2}
\end{equation}
for all $x\in\M(\Theta_{q}^{\sigma})$. We will apply Theorem \ref{thm:pointwise-ergodic}
to both sequences $\{\psi_{(\pm,\delta_{n})}^{n}\}_{n\in\N}$ together
with the sequence $\{t_{n}\}_{n\in\N}$. We choose $t_{n}$ such that
\[
t_{n}\stackrel{\mathrm{def}}{=}\frac{2}{\kappa'}\log n
\]
where $\kappa'$ is the constant from Theorem \ref{thm:pointwise-ergodic},
so that $\sum_{n}e^{-\kappa't_{n}}=\sum_{n}n^{-2}$ is finite. We
choose
\[
\delta_{n}\stackrel{\mathrm{def}}{=}e^{-\kappa't_{n}/4}.
\]

Therefore we may apply Theorem \ref{thm:pointwise-ergodic} to both
sequences $\{\psi_{(\pm,\delta_{n})}^{n}\}_{n\in\N}$ to obtain 2
sets of full measure, $\M'_{\pm}$, such that for all $x'\in\M'\stackrel{\mathrm{def}}{=}\M'_{+}\cap\M'_{-}$,
there exists $n_{x'}\geq0$ such that for all $q$ coprime to $Q_0$,
for all $n\geq n_{x'}$, for all $x\in\cover_{\Theta_{q}^{\sigma}}^{-1}(x')$, 
\begin{equation}
|\Ave_{K,t_{n}}[{\widehat{\psi_{(\pm,\delta_{n})}^{n}}}^{\xi}](x)-|G_{\sigma}(q)|^{-1}\mu_{q}^{\sigma}({\widehat{\psi_{(\pm,\delta_{n})}^{n}}}^{\xi})|\leq e^{-\kappa't_{n}}q^{\alpha}\sqrt{\|\psi_{(\pm,\delta_{n})}^{n}\|_{\infty}^{2}+\|\partial_{\theta}\psi_{(\pm,\delta_{n})}^{n}\|_{\infty}^{2}}.\label{eq:temp1}
\end{equation}
Applying (\ref{eq:banach-bound-for-smoothed-functions}) to the right
hand side gives that it is
\begin{equation}
\leq e^{-\kappa't_{n}}q^{\alpha}\sqrt{1+O(\delta_{n}^{-2})}=O(q^{\alpha}e^{-\kappa't_{n}/2}).\label{eq:temp2}
\end{equation}
Also the Siegel-Veech formula (Theorem \ref{thm:siegelveechformula})
together with (\ref{eq:integral-of-approximant}) gives that 
\begin{align}
\mu_{q}^{\sigma}({\widehat{\psi_{(\pm,\delta_{n})}^{n}}}^{\xi}) & =|\Stab_{G_{\sigma}(q)}(\xi)|c(\CC)\int_{\R^{2}}\psi_{(\pm,\delta_{n})}^{n}d\Leb\label{eq:temp3}\\
 & =|\Stab_{G_{\sigma}(q)}(\xi)|c(\CC)\left(e^{2t_{n}}\theta_{t_{n}}+O(\delta_{n}^{1/2}e^{2t_{n}}\theta_{t_{n}})\right)\\
 & =|\Stab_{G_{\sigma}(q)}(\xi)|c(\CC)\left(e^{2t_{n}}\theta_{t_{n}}+O(e^{(2-\kappa'/8)t_{n}}\theta_{t_{n}})\right).\label{eq:temp4}
\end{align}
Putting (\ref{eq:temp1}), (\ref{eq:temp2}), and (\ref{eq:temp4})
together, along with the orbit-stabilizer theorem, gives that (under
the same conditions on $x,n,q$)
\[
\Ave_{K,t_{n}}[{\widehat{\psi_{(\pm,\delta_{n})}^{n}}}^{\xi}](x)=\frac{c(\CC)}{|G_{\sigma}(q).\xi|}\left(e^{2t_{n}}\theta_{t_{n}}+O(e^{(2-\kappa'/8)t_{n}}\theta_{t_{n}})\right)+O(q^{\alpha}e^{-\kappa't_{n}/2}),
\]
so from (\ref{eq:sandwich2})
\begin{equation}
N^{\CC}(x;e^{t_{n}},\xi)=\frac{\pi c(\CC)}{|G_{\sigma}(q).\xi|}e^{2t_{n}}+O(e^{(2-\kappa'/8)t_{n}})+O(q^{\alpha}e^{(2-\kappa'/2)t_{n}})\label{eq:uninterpolated}
\end{equation}
where in the second term on the right hand side we used simply $|G_{\sigma}(q).\xi|\geq1$
and in the final term we used $\theta_{t_{n}}^{-1}\ll e^{2t_{n}}$.

To conclude the argument we must interpolate (\ref{eq:uninterpolated})
to values of $t$ with $t_{n-1}\leq t\leq t_{n}$. This follows as
$N^{\CC}(x;e^{t},\xi)$ is monotone increasing in $t$ and, for example,
to compare the main term at $t_{n}$ and $t_{n-1}$, we have 
\[
e^{2t_{n}}-e^{2t_{n-1}}=n^{4/\kappa'}-(n-1)^{4/\kappa'}\ll n^{4/\kappa'}.n^{-1}=e^{(2-\kappa'/2)t_{n}}.
\]
These differences can be absorbed into the error term.
\end{proof}

\vspace*{\fill}
\noindent Michael Magee, \\Department of Mathematical Sciences, \\ Durham University,\\  Durham DH1 3LE, U.K. \\  {\tt michael.r.magee@durham.ac.uk}\\ \\
\noindent Rene R\"{u}hr, \\ Department of Mathematics,\\ Tel Aviv University, \\ Israel \\ {\tt rene@post.tau.ac.il }\\ \\

%%%%%%%%%%%%%%%%%%%%%%%%% APPENDIX
\newpage
\appendix

\section{By Rodolfo Guti\'{e}rrez-Romo}
\label{ap}

	The purpose of this appendix is to prove Proposition  \ref{prop:Strong-Approximation}, that is, the Strong Approximation (SA) hypothesis for Rauzy--Veech monoids.

	Recall that there is a ``canonical'' way to obtain a translation surface from a permutation $\pi$. Namely, we consider the translation surface with length data $\lambda_\alpha = 1$ and suspension data $\tau_\alpha = \pi_{\mathrm{b}}(\alpha) - \pi_{\mathrm{t}}(\alpha)$ for each $\alpha \in \mathcal{A}$. We denote the polygon induced by these data by $P_\pi \subseteq \mathbf{C}$. Moreover, recall that $\Lambda_\pi \leq \mathrm{PMod}(S_\pi, \Sigma_\pi)$ is the Rauzy--Veech monoid associated with the permutation $\pi$. Let $\Xi_{\pi,q}^\sigma \leq \mathrm{Aut}(H_1^\sigma(\mathbf{Z}_q))$ be the group generated by $\Theta_q^\sigma(\Lambda_\pi)$ (with $q$ possibly equal to $\infty$). Let $\Gamma(q)$ be kernel of the action of $\Gamma$ on $H_1(\mathbf{Z}_q)$.
	
	We start by adapting the classification of Rauzy--Veech groups \cite{AMY, GR} to our context. Indeed, for each permutation $\pi$ there exist natural maps
	\[
		H_1(S_\pi \setminus \Sigma_\pi, \mathbf{Z}) \to H_1(S_\pi, \mathbf{Z}) \to H_1(S_\pi, \Sigma_\pi, \mathbf{Z}),
	\]
	the former being surjective and the latter injective. We identify $H_1(S_\pi, \mathbf{Z}) = H_1(\mathbf{Z})$ with the image of the second map.
	
	The Rauzy--Veech group induces a left action on $H_1(S_\pi, \Sigma_\pi) = H_1^{\mathrm{rel}}(\mathbf{Z})$ that preserves $H_1^\sigma(\mathbf{Z})$ \cite[Section 6]{GR}. The Rauzy--Veech group $\Xi_{\pi,\infty}^\sigma$ is precisely the groups of symplectic automorphisms induced by the restriction of this action to $H_1^\sigma(\mathbf{Z})$. Thus, we have that:
	
	\medskip
	\begin{thm}
		The Rauzy-Veech group $\Xi_{\pi,\infty}^0$ contains $\Theta_\infty^0(\Gamma(2))$ \cite[Theorem 1.1]{AMY}\cite[Theorem 1.1]{GR}.
	\end{thm}
	\medskip

	We now fix an odd integer $q > 2$ until the end of the section. We obtain the following corollary:
	
	\begin{cor} \label{cor:gamma2}
		 We have that $\Xi_{\pi,q}^0 = G(q)$.
	\end{cor}
	
	\begin{proof}
		We will show that $\Theta_q^0(\Gamma(2)) = G(q)$. This is a consequence of the following two facts: every $p$-th power of a Dehn twist belongs to $\Gamma(p)$, and $\Gamma$ is generated by Dehn twists. Indeed, if $c$ is a simple closed curve then $(T_c^2)^{(q+1)/2} \in \Gamma(2)$. Since $T_c^q \in \Gamma(q)$, we conclude that $\Theta_q^0(T_c) = \Theta_q^0(T_c^{q+1}) = \Theta_q^0((T_c^2)^{(q+1)/2}) \in \Theta_q^0(\Gamma(2))$.
	\end{proof}
		
	We will also need the following definitions originally introduced by Avila and Viana \cite[Section 5]{AV}.
	
	\begin{defn}
		Let $\pi$ be an irreducible permutation on an alphabet $\mathcal{A}$. Let $\alpha \in \mathcal{A}$ and let $\pi'$ be the permutation on $\mathcal{A} \setminus \{\alpha\}$ obtained by erasing the letter $\alpha$ from the top and bottom rows of $\pi$. If $\pi'$ is irreducible, we say that it is a \emph{simple reduction} of $\pi$.
	\end{defn}
	
	\begin{defn}
		Let $\pi'$ be an irreducible permutation on an alphabet $\mathcal{A}'$ not containing $\alpha$. Let $\beta, \beta' \in \mathcal{A}'$ such that $(\alpha_{\mathrm{t}, 1}, \alpha_{\mathrm{b}, 1}) \neq (\beta, \beta')$. We define the permutation $\pi$ on the alphabet $\mathcal{A}' \cup \{\alpha\}$ by inserting $\alpha$ just before $\beta$ in the top row and just before $\beta'$ in the bottom row of $\pi'$. We say that $\pi$ is a \emph{simple extension} of $\pi'$. 
	\end{defn}
	
	These notions allow us to transfer information about Rauzy--Veech monoids between Rauzy classes associated with different strata by adding or removing letters. This is made precise by the following lemma \cite[Lemma 2.4, Lemma 2.7, Lemma 6.3]{GR}:
	
	\begin{lem} \label{lem:copy}
		The Rauzy--Veech group $\Xi_{\pi,\infty}^0$ contains the Dehn twists along the curves $\gamma_\alpha$ joining the two $\alpha$-sides of the polygon $P_\pi$ for each $\alpha \in \mathcal{A}$, oriented upwards. Moreover, these Dehn twists generate $\Xi_{\pi,\infty}^0$. In particular, if $\pi'$ is a permutation on an alphabet $\mathcal{A}' \subsetneq \mathcal{A}$ and $\pi$ is obtained from $\pi'$ by a sequence of genus-preserving simple extensions, then the group generated by $\{T_{\gamma_\alpha}\}_{\alpha \in \mathcal{A}'}$ contains $\Theta_\infty^0(\Gamma(2))$.
	\end{lem}
	
	We need some control on the way in which the singularities are split by simple extensions. When we speak about ``minimal strata'', we refer to strata with only one singularity.
	\begin{lem} \label{lem:insertions}
		Let $\mathcal{C}$ be a connected component of a stratum of the moduli space of genus-$g$ translation surface having $n$ singularities, for $n \geq 2$. There exists a permutation $\pi'$ on an alphabet $\mathcal{A}'$ representing some component of a minimal stratum, and a permutation $\pi$ on an alphabet $\mathcal{A}$ representing $\mathcal{C}$ which is obtained from $\pi'$ by a sequence of genus-preserving simple extensions. Moreover, a lift of $\sigma$ belongs to the $\mathbf{Z}$-submodule generated $\{e_\alpha\}_{\alpha \in \mathcal{A} \setminus \mathcal{A}'}$.
	\end{lem}
	\begin{proof}
		Let 
		{\small\[
			\tau_d =
			\begin{pmatrix}
				1 & 2 & 3 & \cdots & d-5 & d-4 & d-3 & d-2 & d-1 & d \\
				d & d-1 & d-2 & \cdots & 6 & 3 & 2 & 5 & 4 & 1
			\end{pmatrix}
		\]}
		for $d \geq 6$ and
		{\small\[
			\sigma_d =
			\begin{pmatrix}
				1 & 2 & 3 & \cdots & d-7 & d-6 & d-5 & d-4 & d-3 & d-2 & d-1 & d \\
				d & d-1 & d-2 & \cdots & 8 & 3 & 2 & 7 & 6 & 5 & 4 & 1
			\end{pmatrix}
		\]}
		for $d \geq 8$. For even $d = 2g$, we have that $\tau_d$ represents $\mathcal{H}(2g - 2)^{\mathrm{odd}}$ and that $\sigma_d$ represents $\mathcal{H}(2g - 2)^{\mathrm{even}}$ if $d \mod 8 \in \{0, 6\}$, and that the spin parities are reversed if $d \mod 8 \in \{2, 4\}$ \cite[Lemma 4.1]{GR}.
	
		Let $\mathcal{A}' = \{1, \dotsc, 2g\}$. We will consider several cases:
		\begin{itemize}
			\item If $\mathcal{C} = \mathcal{H}(g - 1, g - 1)^{\mathrm{hyp}}$, we take
			\[
				\pi' = \begin{pmatrix}
					1 & 2 & \cdots & 2g \\ 2g & 2g - 1 & \cdots & 1
				\end{pmatrix}, \quad 
				\pi = \begin{pmatrix}
					1 & 2g+1 & 2 & \cdots & 2g \\ 2g & 2g - 1 & \cdots & 2g+1 & 1
				\end{pmatrix}.
			\]
			\item If $\mathcal{C}= \mathcal{H}(g - 1, g - 1)^{\mathrm{nonhyp}}$ for $g \geq 4$, we take $\pi' = \tau_{2g}$ and insert a new letter before $2g$ in the top row and after $2g$ in the bottom row \cite[Lemma 4.1]{GR}.
			\item If $\mathcal{C} = \mathcal{H}(2m_1, \dotsc, 2m_n)^{\mathrm{spin}}$, where $\mathrm{spin} \in \{\mathrm{even}, \mathrm{odd}\}$, we take $\pi' \in \{\sigma_{2g}, \tau_{2g}\}$ having the same spin parity as $\mathcal{C}$. We can insert letters iteratively using the following two facts: genus-preserving simple extensions preserve the spin parity \cite[Lemma 6.4]{GR}, and singularities can be split in any way using simple extensions \cite[Lemma 6.5]{GR}.
			
			\noindent To ensure that the lift of $\sigma$ belongs to the $\mathbf{Z}$-submodule generated by the sides associated with the new letters, we move along the strata in the following way (each arrow represents a simple extension):
			\[
				\mathcal{H}(2g - 2) \to \mathcal{H}(2m_1, 2g - 4) \to \mathcal{H}(2m_1, 2m_2, 2g - 6) \to \dotsb {}
			\]
			Moreover, every letter is inserted before the letter $2$ in the top row, and in a way such that the left endpoint of the $\alpha$-side of $M_\pi$ is the singularity of order $2m_k$, and its right endpoint is the singularity of order $2g - 2 - 2k$, where $\alpha$ is the letter inserted at the $k$-th step. These conditions imply that the set $\{\delta(e_\alpha)\}_{\alpha \in \mathcal{A} \setminus \mathcal{A}'}$ is linearly independent, so a lift of $\sigma$ belongs to the $\mathbf{Z}$-submodule generated by $\{e_\alpha\}_{\alpha \in \mathcal{A} \setminus \mathcal{A}'}$.
			\item If $\mathcal{C}$ represents a connected stratum, then we insert letters in the same way as in the previous case, starting from either $\tau_{2g}$ or $\sigma_{2g}$.
			
			\end{itemize}
	\end{proof}
	
	We can now prove the proposition:
	\begin{proof}[Proof (Proposition 2.10)]
		
		Let $\mathcal{C}$ be a connected component of a stratum of the moduli space of genus-$g$ translation surfaces. Let $\pi'$ and $\pi$ be the permutations in \Cref{lem:insertions}. Thus, we have that a lift $\sigma'$ of $\sigma$ belongs to the $\mathbf{Z}$-submodule generated by $\{e_\alpha\}_{\alpha \in \mathcal{A} \setminus \mathcal{A}'}$.
		
		Recall that there exists a (non-canonical) isomorphism $G_\sigma(q) \to H_1(\mathbf{Z}_q) \rtimes G(q)$ defined by $g \mapsto (g\sigma_q' - \sigma_q', \beta(g))$, where $\beta$ is the surjection in the short exact sequence \eqref{eq:group-SES}. We will identify $G_\sigma(q)$ with $H_1(\mathbf{Z}_q) \rtimes G(q)$.
		
		Let $T_{\gamma_\alpha} \in \Xi_{\pi,\infty}^\sigma \leq \mathrm{Aut}(H_1^\sigma(\mathbf{Z}))$ be the Dehn twist along the curve $\gamma_\alpha$. By \Cref{lem:copy}, the group generated by $\{T_{\gamma_\alpha}|_{H_1(\mathbf{Z})}\}_{\alpha \in \mathcal{A}'}$ contains $\Theta_\infty^0(\Gamma(2))$. Moreover, $T_{\gamma_\alpha}$ fixes $\sigma'$ if $\alpha \in \mathcal{A}'$, since it fixes every $e_{\alpha'}$ for $\alpha' \neq \alpha$. 
		
		Let $\Xi$ be the group generated by the action of $\{T_{\gamma_\alpha}\}_{\alpha \in \mathcal{A}'}$ on $H_1^\sigma(\mathbf{Z}_q)$. We obtain that $\Xi$ fixes $\sigma'$ and, by \Cref{cor:gamma2}, that it contains the subgroup $0 \rtimes G(q)$ of $H_1(\mathbf{Z}_q) \rtimes G(q)$. Since $\sigma' \neq 0$ and is generated by $\{e_\alpha\}_{\alpha \in \mathcal{A} \setminus \mathcal{A}'}$, there exists $\alpha \in \mathcal{A} \setminus \mathcal{A}'$ such that $T_{\gamma_\alpha}$ does not fix $\sigma_q'$. Moreover, there exists $g \in \Xi$ such that $T_{\gamma_\alpha}|_{H_1(\mathbf{Z})}$ and $g|_{H_1(\mathbf{Z})}$ coincide mod $q$. Therefore, $T_{\gamma_\alpha}g^{-1}$ fixes $H_1(\mathbf{Z}_q)$ and does not fix $\sigma_q'$, so $T_{\gamma_\alpha}g^{-1}$ can be written as $(h, \mathrm{Id}) \in H_1(\mathbf{Z}_q) \rtimes G(q)$, where $h = T_{\gamma_\alpha}\sigma_q' - \sigma_q' \neq 0$. This concludes the proof, since $(h, \mathrm{Id}).(0, g) = (h, g)$ and $G(q)$ acts transitively on $H_1(\mathbf{Z}_q)$.
	\end{proof}
\vspace*{\fill}
\noindent Rodolfo Guti\'{e}rrez-Romo, \\
Institut de Math\'{e}matiques de Jussieu -- Paris Rive Gauche,\\
 UMR 7586, B\^{a}timent Sophie Germain, 75205 \\
PARIS Cedex 13, France.\\ 
 {\tt rodolfo.gutierrez@imj-prg.fr}\\ \\

\newpage 

\def\cprime{$'$}

\end{document}